\documentclass[10pt]{amsart}
\usepackage{geometry}                
\usepackage{graphicx}
\usepackage{amssymb}
\usepackage{hyperref}
\usepackage{float}
\usepackage{subcaption}
\usepackage{eufrak}

\theoremstyle{plain} 
\newtheorem{theorem}{Theorem}[section]
\newtheorem{lemma}[theorem]{Lemma}

\newtheorem{proposition}[theorem]{Proposition}
\theoremstyle{definition} 
\newtheorem{definition}[theorem]{Definition}
\newtheorem{remark}[theorem]{Remark}

\newcommand{\Wcal}{\mathcal{T}}

\newcommand{\R}{\mathbb{R}}

\newcommand{\fF}{\mathfrak{F}}

\begin{document}
\begin{titlepage}

\begin{title}
{Biplanar Foldings}
\end{title}

\author{
Matthias Weber
}
\address{Matthias Weber\\Department of Mathematics\\Indiana University\\
Bloomington, IN 47405
\\USA}
\author
{Jiangmei Wu}
\address{Jiangmei Wu\\School of Art, Architecture, and Design\\Indiana University\\
Bloomington, IN 47405
\\USA}

\date{\today}

\begin{abstract}
We  introduce a general geometric framework for the construction  of  polyhedra and polyhedral complexes that   are {\em bifoldable}, i.e. foldable into two different planes.
This vastly generalizes origami folds known as the Miura pattern, the origami tube and the Eggbox pattern. Our polyhedra are generalized zonohedra  based on 1-parameter families stars of vectors in $\R^3$ that deform in specific ways while the polyhedra are folded. After describing the framework, its basic features and the general design process, we give several new examples of infinite doubly periodic, triply periodic and fractal bifoldable polyhedra.  
\end{abstract}

\maketitle

\end{titlepage}
\section{Introduction}

	In recent years, origami has found a wide range of applications in material sciences, engineering, cell biology, art and other areas. At a micro scale, origami design has been used in medical devices such as heart stents \cite{Kuribayashi2006}. At a large scale, origami design has been applied in art installation \cite{Wu2018}, aerospace \cite{Barbarino2011}, and architectural facade \cite{Grosso2010}. In more recent years, origami has also been used to create programmable systems that can change shape, functionality and material properties \cite{Hawkes2010}. While many of these applications are concerned with origami that is folded from a flat sheet material to create three-dimensional depth, a few origami researchers have been focusing on creating origami tubes or honeycomb-like origami structures, such as foldable   polyhedral origami that can be repeated periodically in space. Such origami structures can be used as inflatable and deployable structures \cite{Schenk2014} or as metamaterials \cite{filipov2015}. 

\medskip

Several of these structures,  have the remarkable feature that they can be folded together in {\em two different} ways into perpendicular planes.
However, most of the examples given in the literature are general polyhedral complexes, i.e. have edges where more than two facets meet. For construction purposes, it is often desirable to have purely polyhedral structures. Moreover, the fundamental geometric construction methods that are used to generate these origami structures are somewhat limited and appear to be coincidental. Many of the resulting honeycomb origami structures are closely related to Miura folds \cite{Schenk2013} or space-filling polyhedra \cite{Tachi2012}.

We address these two issues by first describing a simple mathematical framework (Section \ref{sec:star}) that explains the bifoldability and allows for a very flexible construction method of general  bifoldable complexes ($\Sigma$-comoplexes) in which no more than two facets meet at any given edge. 
In section \ref{sec:vertices} we list all fourteen types of polyhedral vertices  that can occur in our construction.
After reviewing the basic examples (Section\ref{sec:simple}),  we  illustrate our method to construct a bifoldable fractal (Section \ref{sec:fractal}).
We briefly discuss the design process to build more complicated examples in Section \ref{sec:design}. We then demonstrate this method in Section  \ref{sec:weave} to construct a doubly periodic Miura Weave.   In Section \ref{sec:links} we discuss simple examples of triply periodic bifoldable surfaces with dodecahedral cavities, and in Section \ref{sec:dosequis} we construct a triply periodic Miura pattern. 

We believe our construction method of bifoldable complexes will have implications in the designing and building of smart metamaterials, air or hydraulic filtration systems, robots, large-scale inflatable structures, breathable architectural skins, and many more.  

\medskip

Initially, our investigation began with the question when polygonal approximations of classical triply periodic minimal surfaces can be folded into a plane. A simple and well known example is the mucube \cite{coxeter1938}, an approximation of the P-surface of Schwarz. After discovering Theorem \ref{thm:key}, the focus quickly shifted away from minimal surfaces. In fact, some of our examples like the Fractal are distinctively non-minimal (minimal surfaces have quadratic area growth in balls). On the other hand, we observed three paradigms that triply periodic minimal surfaces and  bifoldable polyhedra have in common: First, there are examples where triply periodic minimal surfaces can be collapsed into different planes through a minimal (albeit not isometric) deformation. Secondly, the flexibility one has in constructing examples is similar. Lastly, certain topological paradigms occur in both situations: The Link has a vague resemblance to Scherk's singly periodic surface, one can create a bifoldable version of Scherk's doubly periodic surface, and the surfaces discussed in section \ref{sec:dosequis} can be modified to resemble the doubly periodic Karcher-Meeks-Rosenberg surfaces.

\section{Stars and Foldings}\label{sec:star}

In this section, we introduce a general framework to construct a class of bifoldable Euclidean polyhedral complexes that we call $\Sigma$-complexes.
To make this notion rigorous, we introduce the following concepts.

\begin{definition}
A (geometric) polyhedral complex in $\R^n$ is a locally finite set $P$ of closed convex polytopes in $\R^n$
such that 
\begin{itemize}
\item with each polytope, all of its facets belong to the set;
\item the intersection between two polytopes is either empty or is a facet of both of them;
\item as a topological space, $P$ is connected.
\end{itemize}
The dimension $d$ of a polyhedral complex is the maximal dimension of its polytopes. We then speak of polyhedral $d$-complexes.
\end{definition}

We will only consider 2- and 3-dimensional polyhedral complexes in $\R^3$.

\begin{definition}
A  polyhedral  2-complex in $\R^3$  is {\em biplanar} if all its edges lie  in  one of the two  orthogonal planes which we will assume to be the two vertical coordinate planes of $\R^3$.
\end{definition}

\begin{definition}
A {\em folding}  is a continuous 1-parameter family of  polyhedral 2-complexes $C_t$ so that the facets remain congruent.
A folding {\em collapses} into a plane for $t\to t_0$  if  the limits of all vertices of $C_t$ exist and lie in that plane.
\end{definition}

Our goal is to construct  biplanar polyhedral foldings that collapse at the end points of the parameter interval into the two vertical coordinate planes.

The polyhedral complexes we consider are  generalizations of {\em zonohedra} \cite{coxeter1973}
and are based on the concept of a {\em 4-star}:

\begin{definition}
A 4-star $\Sigma$  is a set of four vectors $v_1,\ldots v_4\in \R^3$ so that no three of them are linearly dependent. A  polyhedral  complex is {\em based} on a star $\Sigma$ if all of its edge vectors lie in $\Sigma$, up to orientation.
\end{definition}

\begin{figure}[H]
   \centering
   \includegraphics[width=3in]{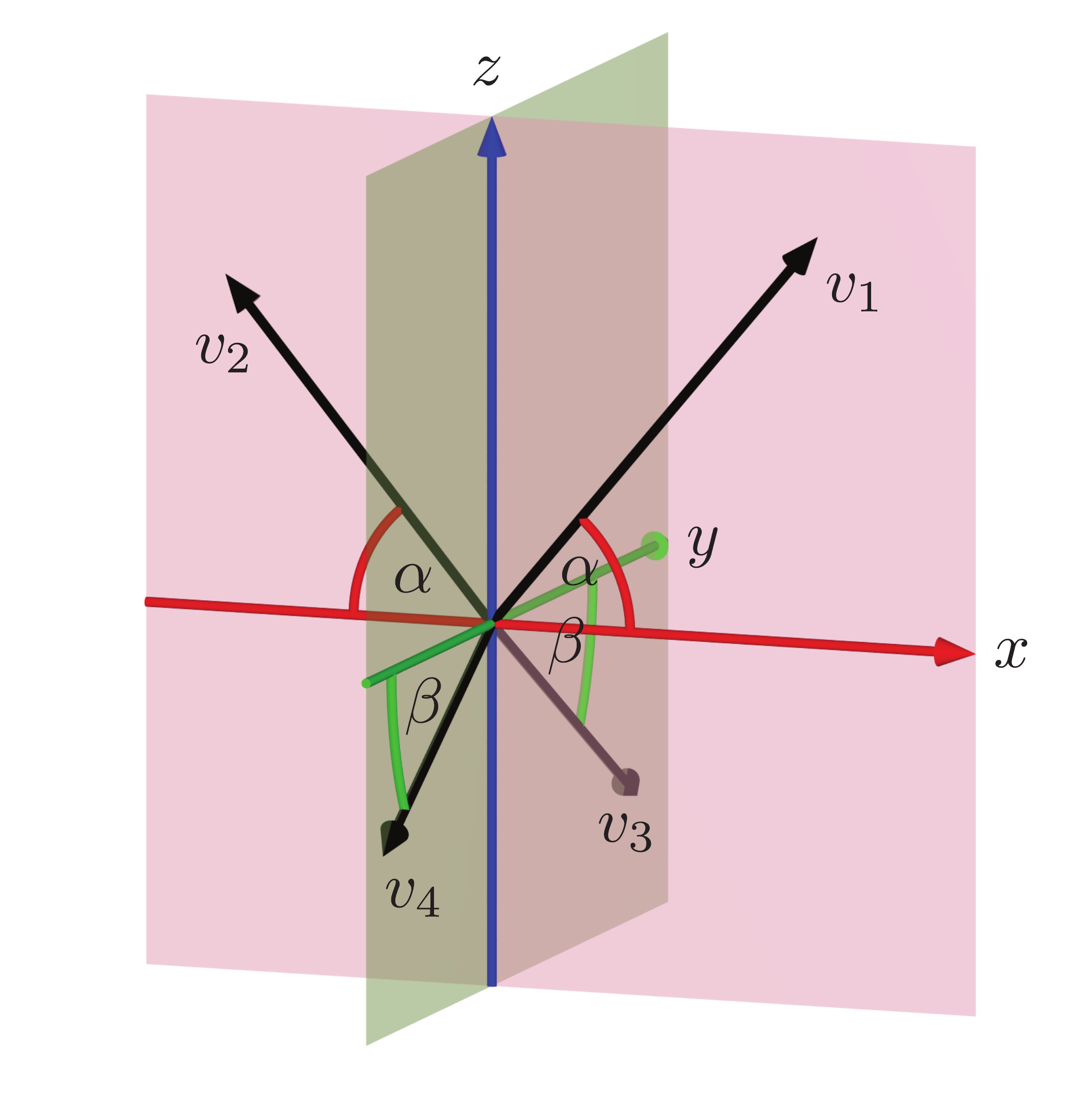}
   \caption{The star $\Sigma$}
   \label{fig:star}
\end{figure}

The most regular star is the {\em tetrahedral star},  given  by 
\begin{align*} 
   v_1 = {}&  (1,1,1)\\
   v_2 = {}&  (1,-1,-1)\\
   v_3 = {}&  (-1,1,-1)\\
   v_4 = {}&  (-1,-1,-1) \ .
\end{align*}
These are the vectors that point from the center to the vertices of a regular tetrahedron placed inside a cube.

There is an abundance of polyhedral complexes based on any given star. To see this, recall that the rhombic dodecahedron tiles space and itself can be tiled in two different ways by four congruent rhomboids. Subdividing each  rhombic dodecahedron of the space tiling one way or the other into four rhomboids gives an infinitude of polyhedral 3-complexes whose 2-facets form polyhedral 2-complexes based on $\Sigma$.
This construction works in fact for any star, not only the tetrahedral star. There are, of course, many more examples, some of which we will encounter below.

We call the six parallelograms $\Pi_{ij}$ (or translational copies of them) that are spanned by $v_i$ and $v_j$,  $i<j$, the {\em facets}.
If a polyhedral complex is based on a star, its facets are necessarily one of these six facets. 
We also denote the plane spanned by $v_i$ and $v_j$ by $v_i\wedge v_j$. The parallelepiped spanned by three different star vectors $v_i$, $v_j$ and $v_k$ will be denoted by $R_{ijk}$.

Our goal is to deform a polyhedral complex that is based on a star by deforming the star. Here we will face two difficulties: The short time {\em existence problem} of deformations, and the long time {\em embeddedness problem} that during a deformation the polyhedral complexes might self-intersect.

We address the existence problem by introducing a homological condition:

\begin{definition}
Let $\eta$ be a path of edges in a polyhedral complex based on a star. We represent this path as a sequence of edges $(\pm v_{i_1},\ldots, \pm v_{i_n})$, where the sign indicates the orientation of the edge. Such a path $\eta$ is closed if $\sum_{k=1}^n \pm v_{i_k} = 0$. We say that $\eta$ is {\em generically closed} if each $v_i$ occurs in $\eta$ equally often with positive and negative orientation. This makes the path closed not only for the star $\Sigma$ but for any choice of star. We call a polyhedral complex {\em generic} if every closed path is generically closed.
\end{definition}

Then we have:

\begin{lemma} 
Suppose a finite polyhedral complex $P$ is generically based on a star $\Sigma$. Let $\Sigma_t$ be a continuous deformation of $\Sigma$. Then, for $t$ small enough, there is a continuous family $P_t$ of polyhedral complexes generically based on $\Sigma_t$.  
\end{lemma}
\begin{proof}
We begin by defining the vertex set of $P_t$ in $\R^3$.
Let $p$ be a fixed vertex of $P$. We define the corresponding vertex of $P_t$  by letting $p_t=p$. Now let $q$ be an arbitrary vertex of $P$, and $e_1,\ldots e_n$ a path of edges from $p$ to $q$ in $P$. As $P$ is based on $\Sigma$, each edge $e_k$ corresponds to a star vector $\sigma_k=\pm v_k$, where $v_k\in \Sigma$ and the sign depends on the orientation of the edge. Then   $q=p+\sum \sigma_k$. We define $q_t = p_t +\sum \sigma_k(t)$, where $\sigma_k(t)=\pm v_k(t)$,  $v_k(t)\in \Sigma_t$ and the  signs are chosen consistently. This definition is independent of the chosen edge path  from 
$p$ to $q$, as we assumed that $P$ is generic. 

The new vertices depend continuously on $t$. By definition of a polyhedral complex, a facet of $P$ is the convex hull of some of its vertices. We use the same vertex sets to define the facets of $P_t$. This will be possible for small $t$ as the intersection condition for the polytopes is an open condition. Here we use the finiteness assumption: It is conceivable that further and further away polytopes need smaller and smaller open neighborhoods in order to remain disjoint.
\end{proof}

The second problem about the long-term existence of deformations is more subtle, as it can easily happen that vertices or edges from different  polytopes become incident during the deformation (see Remark \ref{rem:embed}). There is a special situation that applies to most of our examples where this can be avoided. This occurs when the 2-complex in question is in fact part of the 2-skeleton of an infinite 3-complex without boundary:

\begin{lemma}\label{lem:embed}
Suppose that for $t\in[0,1)$, $P_t$ is a continuous family of finite polyhedral 3-complexes satisfying the following conditions:
\begin{itemize}
\item Each $P_t$ is a subcomplex of a polyhedral 3-complex without boundary.
\item These 3-complexes without boundary are generically based on a  family $\Sigma_t$ of stars.
\item The stars are continuous even for $t\in[0,1]$.
\end{itemize}
Then $P_1$, defined as above, is a polyhedral 3-complex based on $\Sigma_1$.
\end{lemma}
\begin{proof}
Note that polyhedral 3-complexes based on stars have parallelepipeds as their 3-polytopes. These are non-degenerate, as we have assumed that no three star vectors are coplanar. This holds for the 3-polytopes of the limit star $\Sigma_1$ as well. In order to see that $P_1$ is a polyhedral 3-complex, we have to show that the intersection of any two polytopes is a facet. Near any point $p$ of $P_1$, the polyhedral 3-complexes $P_t$ are the union of non-degenerate parallelepipeds.
These either intersect in facets or are disjoint. The same must hold in for the  limit parallelepipeds, as these are non-degenerate. Hence 
 the polytopes of  $P_1$ are either disjoint or share a facet as well.
\end{proof}

We now come to the construction of bifoldable polyhedral complexes.
Our method will be based on the 1-parameter family of stars $\Sigma= \Sigma(r_1,r_2,r_3,r_4,\alpha,\beta)$ given by the vectors

\begin{align*} 
   v_1 = {}&  r_1 (\cos(\alpha), 0, \sin(\alpha))\\
   v_2 = {}&  r_2 (-\cos(\alpha), 0, \sin(\alpha))\\
   v_3 = {}&  r_3 (0,\cos(\beta), - \sin(\beta))\\
   v_4 = {}&  r_4 (0,-\cos(\beta),  -\sin(\beta)) \ .
\end{align*}
Here the  $r_i$ are arbitrary but fixed positive real numbers, and $\alpha, \beta\in(0,\pi/2)$ are angle parameters that will vary dependent on a single parameter. See Figure \ref{fig:star}.
The tetrahedral star can be recovered (up to similarity) by letting $r_i=1$ and $\alpha=\beta=\arccos(\sqrt{2/3})$.

Usually, we  expect  a polyhedral complex that is based on any given star to be rigid. Our key observation is that if we disallow the usage of two of the six facets, the polyhedral complex  becomes foldable.

More precisely we have:
 
\begin{theorem} 
\label{thm:key}
A polyhedral complex $P$ based on $\Sigma$ is biplanar. If it doesn't have any facets of type  $\Pi_{12}$ and $\Pi_{34}$, it can be collapsed into any of the two vertical coordinate planes.
\end{theorem}

\begin{proof}
That $P$ is biplanar is trivial, as the star vectors lie in either of the two vertical coordinate planes. As $P$ is based on $\Sigma$, every vertex of it is an integral linear combination of star vectors. We will  deform $P$ by changing the angle parameters $\alpha$ and $\beta$, but keeping both the $r_i$ and the combinatorial information of $P$. We need to verify that under this deformation the facets of $P$ remain congruent.

A facet parallel to the plane $v_i\wedge v_j$ of the deformed complex will be congruent to the original facet if and only if the dot products $v_i\cdot v_i$, $v_j\cdot v_j$ and $v_i\cdot v_j$ remains unchanged. The first two dot products remain unchanged as we do not change the length of the star vectors during the deformation. As  facets in the planes $v_1\wedge v_2$ and $v_3\wedge v_4$ are forbidden, the third dot product will be equal to $-r_i r_j \sin(\alpha)\sin(\beta)$.
Thus in order to have a folding of  $P$, we  need to deform $\Sigma$ so that $\sin(\alpha)\sin(\beta)$ remains constant. As $\alpha, \beta\in(0,\pi/2)$ we have $\lambda=\sin(\alpha)\sin(\beta)\in(0,1)$. This shows that we can fold $P$ by letting 
\[
\beta = \arcsin(\lambda/\sin(\alpha))
\]
for $\alpha \in (\arcsin(\lambda),\pi/2)$. Note that when $\alpha=\pi/2$, $P$ has collapsed into the plane $x=0$. When $\alpha=\arcsin(\lambda)$ we have $\beta=\pi/2$, and $P$ has collapsed into the plane $y=0$.
\end{proof}

\begin{definition}
We call the facets $\Pi_{12}$ and $\Pi_{34}$ {\em forbidden facets} and the facets $\Pi_{13}$, $\Pi_{14}$, $\Pi_{23}$ and $\Pi_{24}$ {\em admissible facets}. 
\end{definition}

\begin{definition}
We call a generic polyhedral complex based on $\Sigma$ that does not contain any forbidden facets a $\Sigma$-complex.
\end{definition}

We note that the four admissible facets of $\Sigma$ will be parallelograms with acute angle $\gamma$ given by
\[
\cos\gamma = \sin\alpha\sin\beta \ .
\]
In case that all star vectors have the same length, i.e. that $r_1=r_2=r_3=r_3$, these parallelograms will therefore be congruent rhombi.
Observe that necessarily $0<\gamma<\pi/2$, so that squares can never be facets, but all other rhombi can occur.

As explained above, it is very easy to construct  polyhedral complexes based on $\Sigma$, and by removing the forbidden facets we obtain 
a large number of bifoldable $\Sigma$-complexes. In the following sections, we will focus on polyhedral examples.

\begin{definition}
A polyhedral 2-complex is called a {\em polyhedron} if  each edge belongs to at most two facets and each vertex is a manifold point, i.e. the intersection of the polyhedron with a small ball centered at that vertex is homeomorphic to a disk. A $\Sigma$-polyhedron is a $\Sigma$-complex that is also a polyhedron. 
\end{definition}

\begin{remark}
Note that we made the assumption that the two planes in which the star vectors lie are orthogonal. This is crucial: If they make another angle, Theorem \ref{thm:key} fails, because the congruence condition for the admissible facets then forces $\alpha$ and $\beta$ to be constant.
 \end{remark}
 
 \begin{remark}\label{rem:embed}
It is possible for $\Sigma$-polyhedra to self-intersect during a deformation, see Figure \ref{fig:Intersect}. The only sufficient condition we know that prevents this from happening is given by Lemma \ref{lem:embed}. 
\begin{figure}[h]
      \centering
      \subcaptionbox{Before \label{fig:Intersect1}}
        {\includegraphics[width=0.45\textwidth]{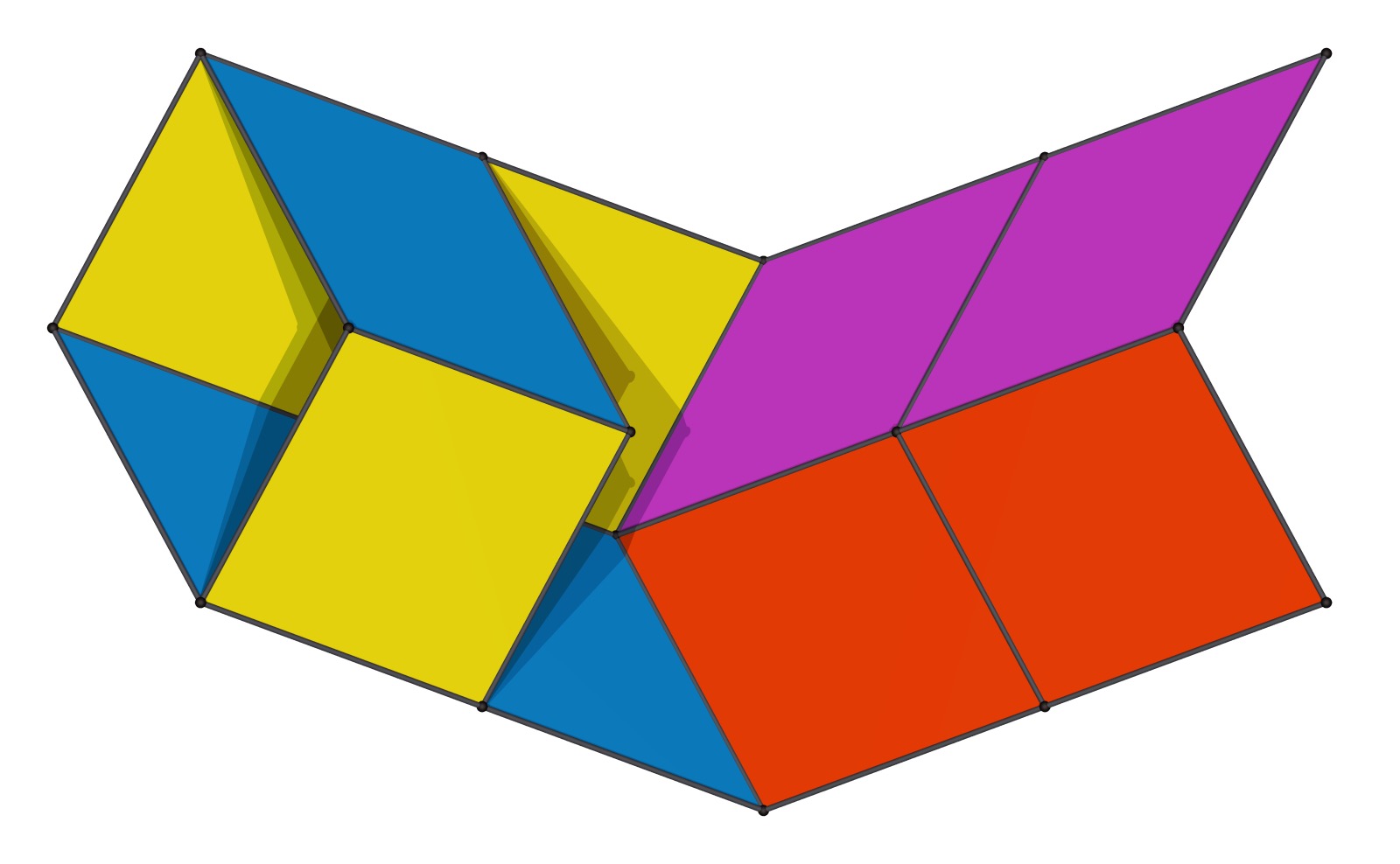}}
      \subcaptionbox{After \label{fig:Intersect2}}
        {\includegraphics[width=0.45\textwidth]{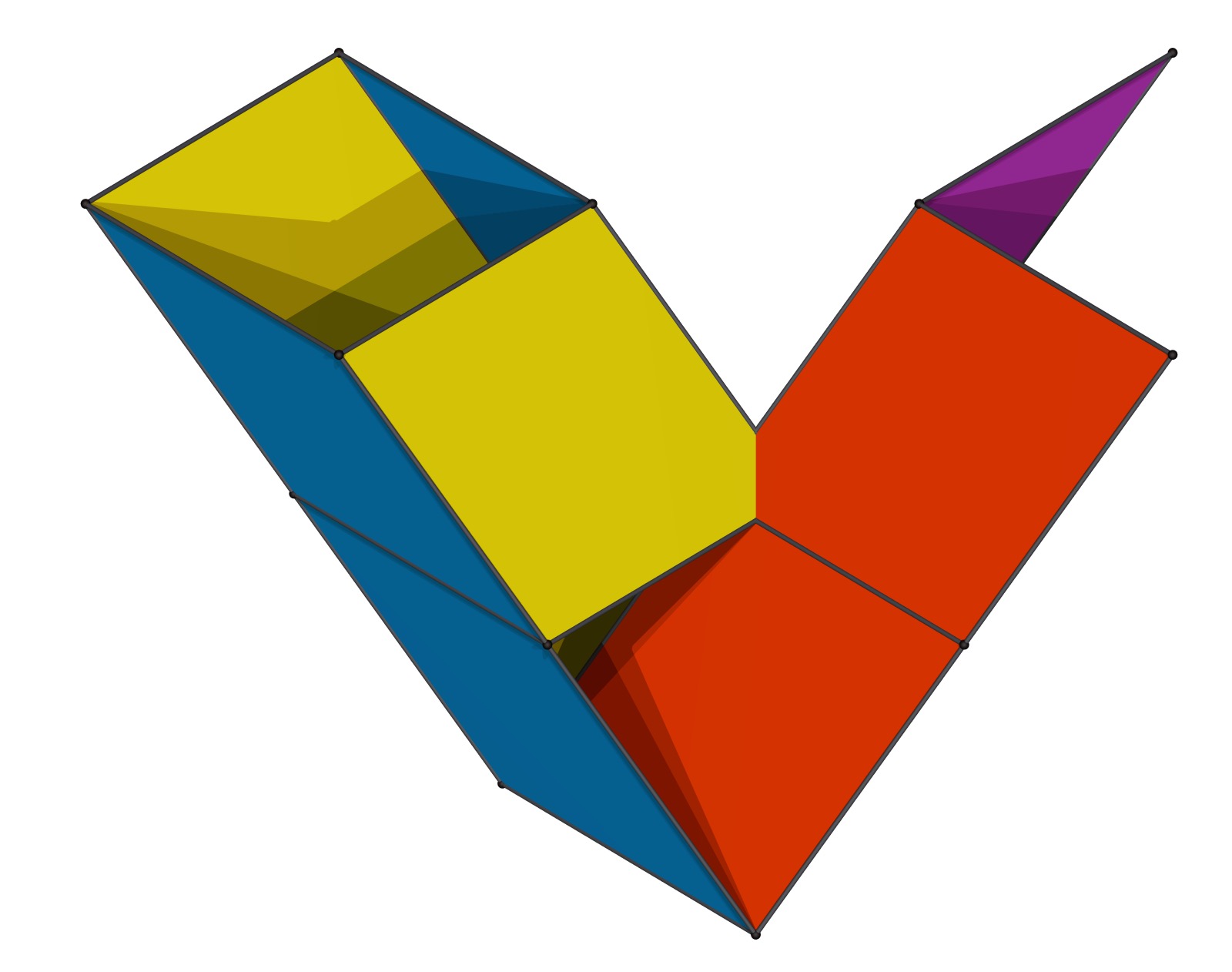}}
      \caption{Intersection during folding}
      \label{fig:Intersect}
    \end{figure}
\end{remark}
 
We conclude this section by showing:

\begin{theorem} \label{thm:finite}
There is no finite $\Sigma$-polyhedron without boundary.
\end{theorem}

\begin{figure}[H]
   \centering
   \includegraphics[width=3in]{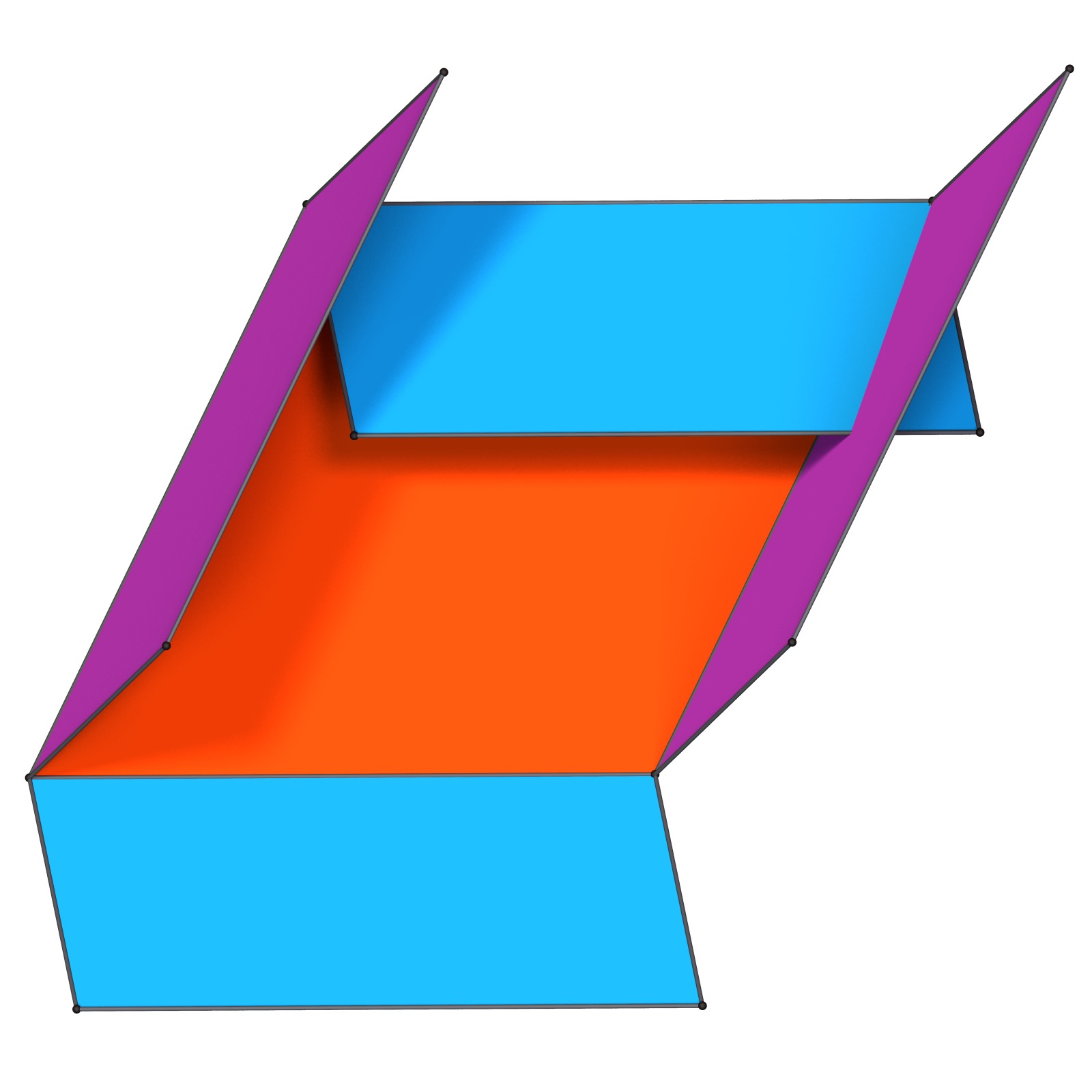}
   \caption{Intersecting side panels}
   \label{fig:FourWays}
\end{figure}

\begin{proof}
We will exploit the fact that $\Sigma$-polyhedra without boundary are (generalized) zonohedra: A {\em zone} is a sequence of facets so that consecutive facets of the sequence share a fixed star vector. Each facet defines two different zones. In general, zones can contain infinitely many facets, but for closed zonohedra without boundary, this number is obviously finite, and the zone is topologically a cylinder. We now assume that we have a finite $\Sigma$-polyhedron, rotate it so that one type of facet is horizontal, and choose  a lowest facet of that type.  We can also assume that the facet in question is of type $\Pi_{13}$.

The zone through this facet with common edge vector $v_1$ consists only of facets of type $\Pi_{13}$ and type $\Pi_{14}$. As there are no facets of type $\Pi_{13}$ below the selected facet, its neighbors are either copies of $\Pi_{13}$, or type $\Pi_{14}$ facets pointing up. A similar statement holds for the other zone with common edge vector $v_3$. Thus near the selected facet, the polyhedron consists of a horizontal finite polygon made up of copies of $\Pi_{13}$ and bounded by facets of type  $\Pi_{14}$ and $\Pi_{23}$ all pointing up.

Now consider a single $\Pi_{13}$  bounded by facets of type  $\Pi_{14}$ and $\Pi_{23}$ all pointing up, as shown in Figure \ref{fig:FourWays} (where ``up'' means ``towards the viewer''). The two $\Pi_{14}$ bound a parallelepiped $R_{134}$ and the two $\Pi_{234}$ another parallelepiped $R_{123}$. The two parallelepipeds share their bottom facets and hence some of their other facets intersect.
They cannot be equal, because this would imply $v_2=v_4$. Thus, for at least for one vertex of $\Pi_{13}$, the two adjancent $\Pi_{14}$ and $\Pi_{23}$ facets intersect. Such a vertex must also occur for the bottom polygon made up of copies of $\Pi_{13}$ and bounded by facets of type  $\Pi_{14}$ and $\Pi_{23}$ all pointing up. Hence the polyhedron has self-intersections, a contradiction.
\end{proof}

\section{Vertex Types}
\label{sec:vertices}

In this section, we list the 14 possible vertex types, i.e. the ways in which the four admissible facets can be grouped around a single vertex in a $\Sigma$-polyhedron.
We have arrived at this list through an exhaustive and systematic enumeration, which we do not reproduce here. 

For each case, we denote by $(a,b)$ the number $a$ of acute angles and the number $o$ of obtuse angles that occur. The sum $a+o$ is the valency of the vertex, and $\kappa = 2\pi -a\gamma -o(\pi-\gamma)$ is the (Gauss) curvature associated to the vertex. The latter quantity is useful to determine the genus $g$ of a closed  polyhedral surface, because the Gauss-Bonnet formula states that
\[
2-2g = \sum_v \kappa_v \ ,
\]
where the sum is taken over all vertices $v$. 

The only possible valencies that can occur are 4, 6, and 8. 
We prove without relying on the enumeration below:

\begin{proposition}
For a $\Sigma$-polyhedron, the valency of a vertex is even.
\end{proposition}
\begin{proof}
The sequence $\Pi_{ij}$ of facets around a vertex gives rise to a sequence of pairs $(ij)$ of indices of edges. The pairs that can occur belong to the set 
\[
\Wcal = \{(13), (31), (14), (41), (23), (32), (24), (42) \} \ .
\]
We distinguish here for once $(ij)$ from $(ji)$ to indicate the order in which the edges occur when following the facets around the vertex. Two pairs can be adjacent in such a sequence if and only if the last index of the first pair equals the first index of the second pair. From a given sequence, we now eliminate pairs of the form $(ij)(ji)$, thereby reducing the number of pairs by an even number. We are then left with a sequence that periodically contains  repetitions of $(13)(32)(24)(41)$, its reversion,  or any of its cyclic permutations, because when  $(ij)(ji)$ pairs are eliminated, any pair determines is successors uniquely. 
\end{proof}

The fact that the valency has to be even follows from the genericity of the stars: The boundary of a vertex figure has as edges the vectors $v_i$, and as the sum has to be zero, each $v_i$ has to occur an even number of times.

We begin with valency 4:

\begin{figure}[H]
\centering
\subcaptionbox{Unfold $(2,2)$, $\kappa=0$  \label{fig:Unfold}}
{\includegraphics[width=.32\linewidth]{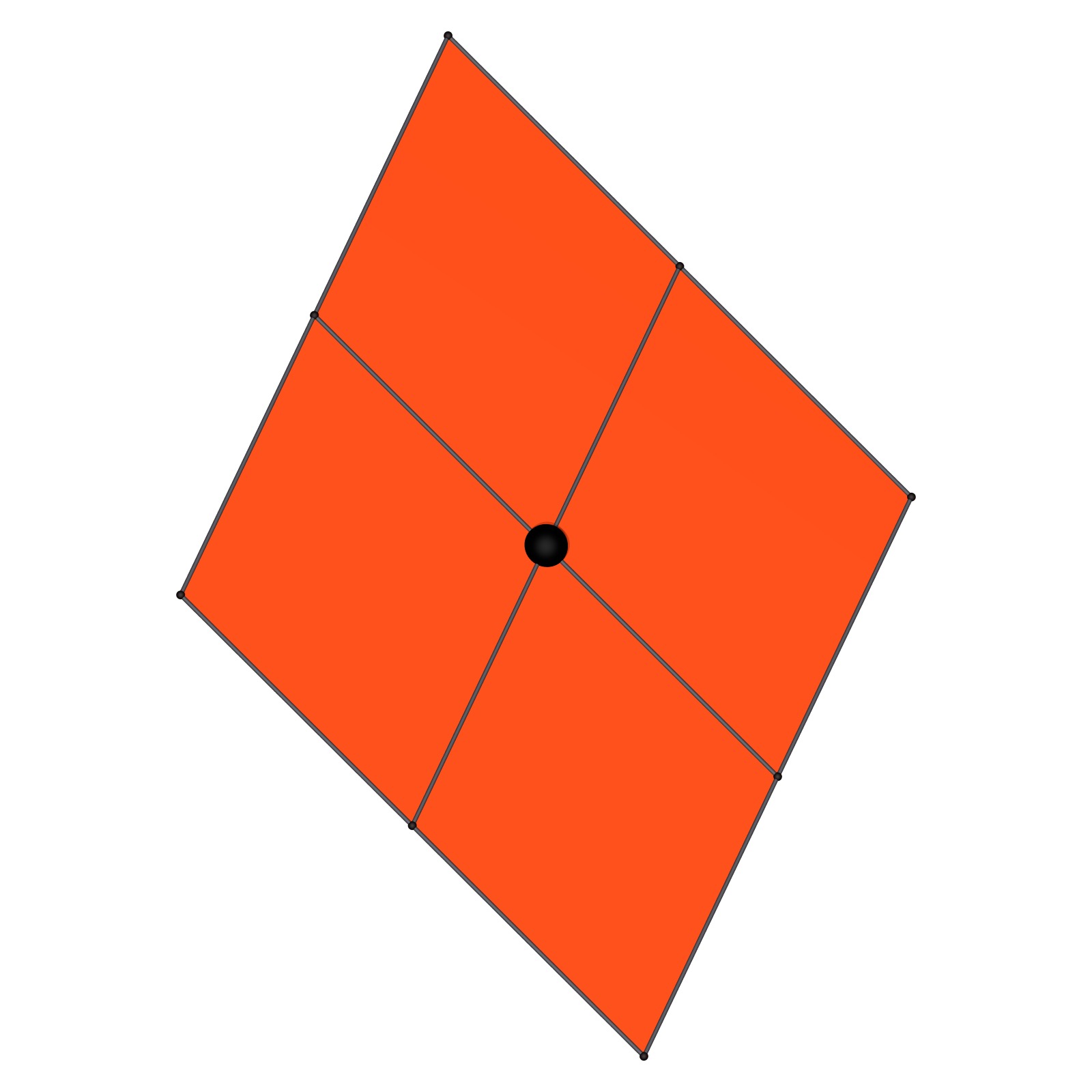}}
\subcaptionbox{Obtuse $(2,2)$, $\kappa=0$\label{fig:FlatObtuse}}
{\includegraphics[width=.32\linewidth]{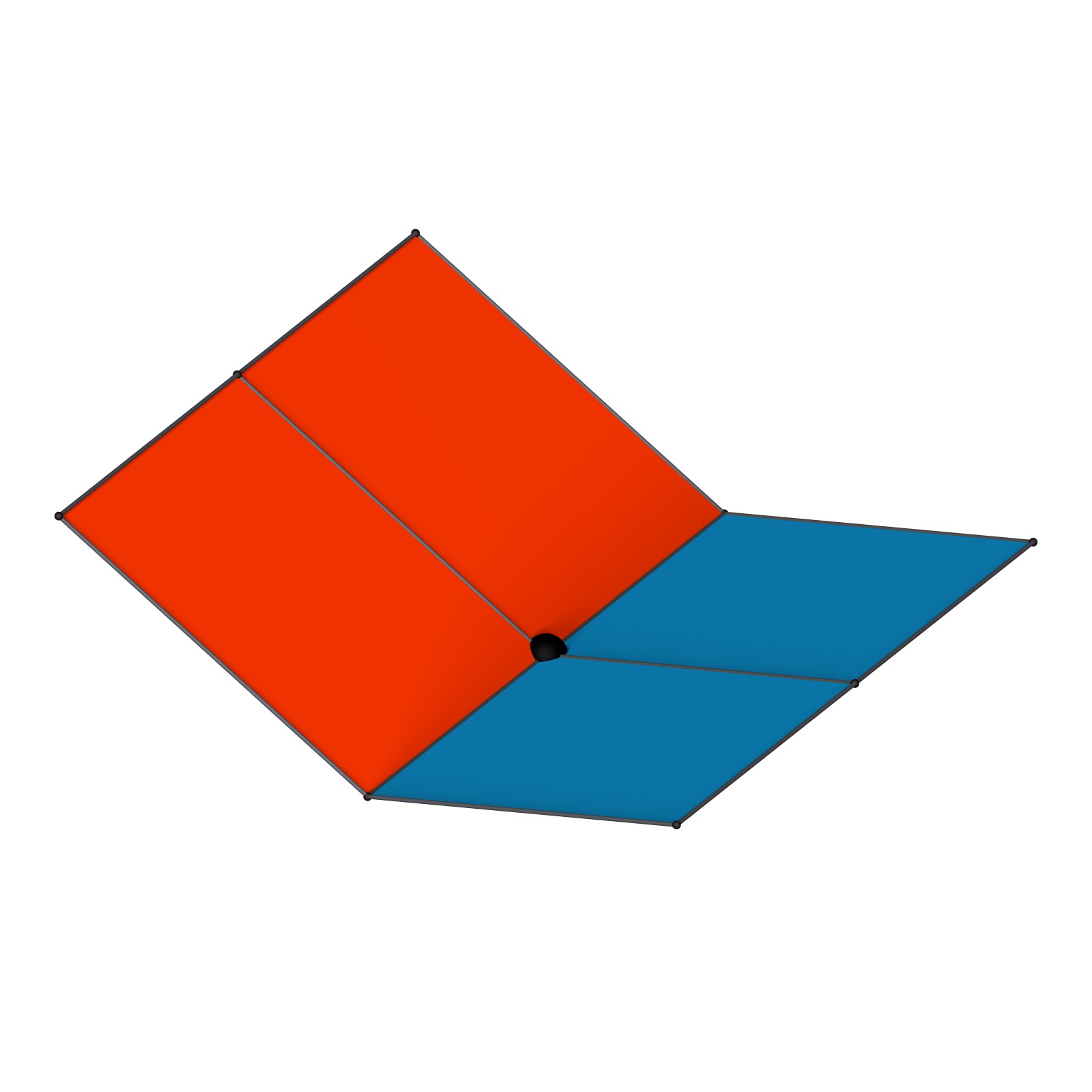}}
\subcaptionbox{Acute $(2,2)$, $\kappa=0$\label{fig:FlatAcute}}
{\includegraphics[width=.32\linewidth]{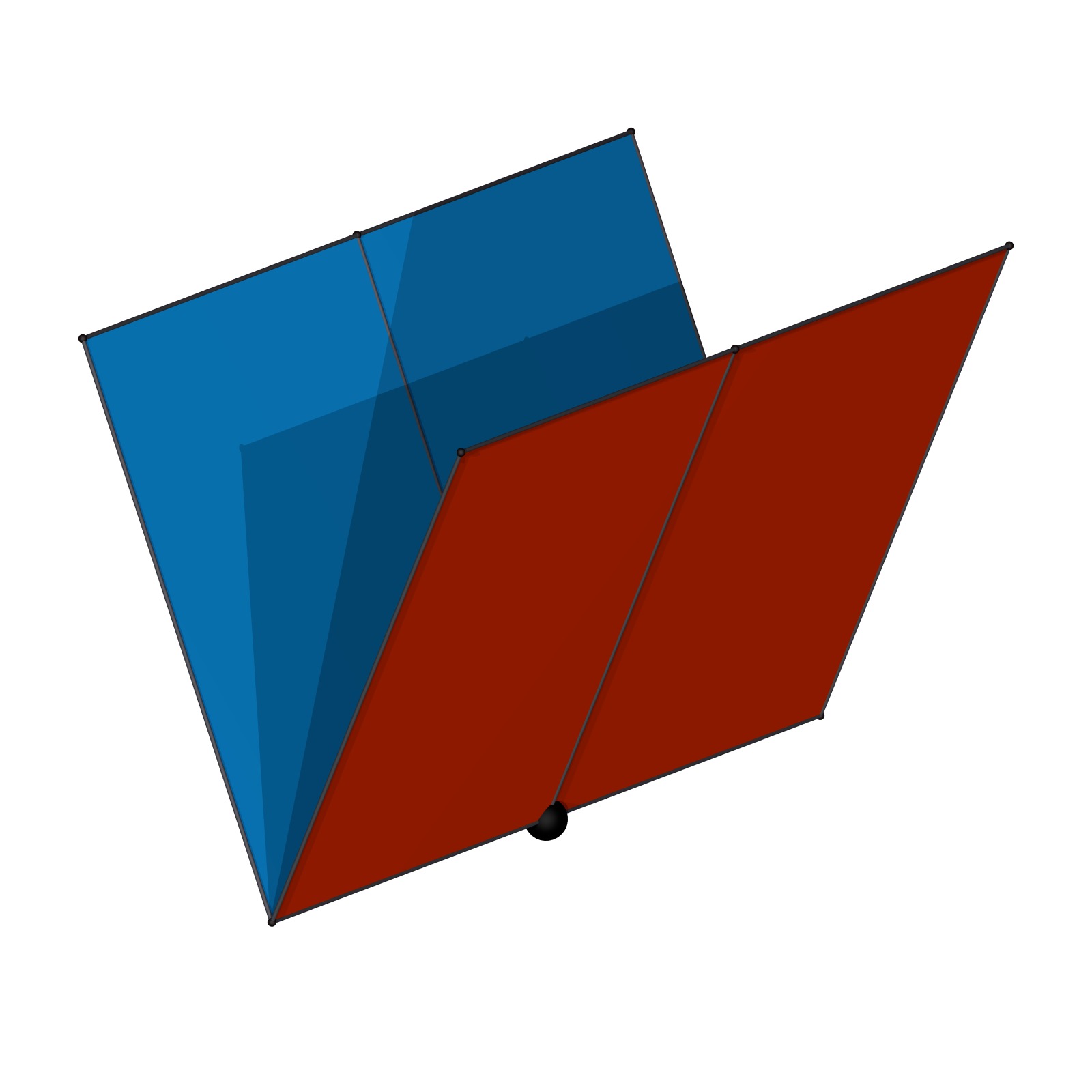}}
\subcaptionbox{Saddle $(0,4)$, $\kappa=4\gamma-2\pi$\label{fig:Saddle}}
{\includegraphics[width=.32\linewidth]{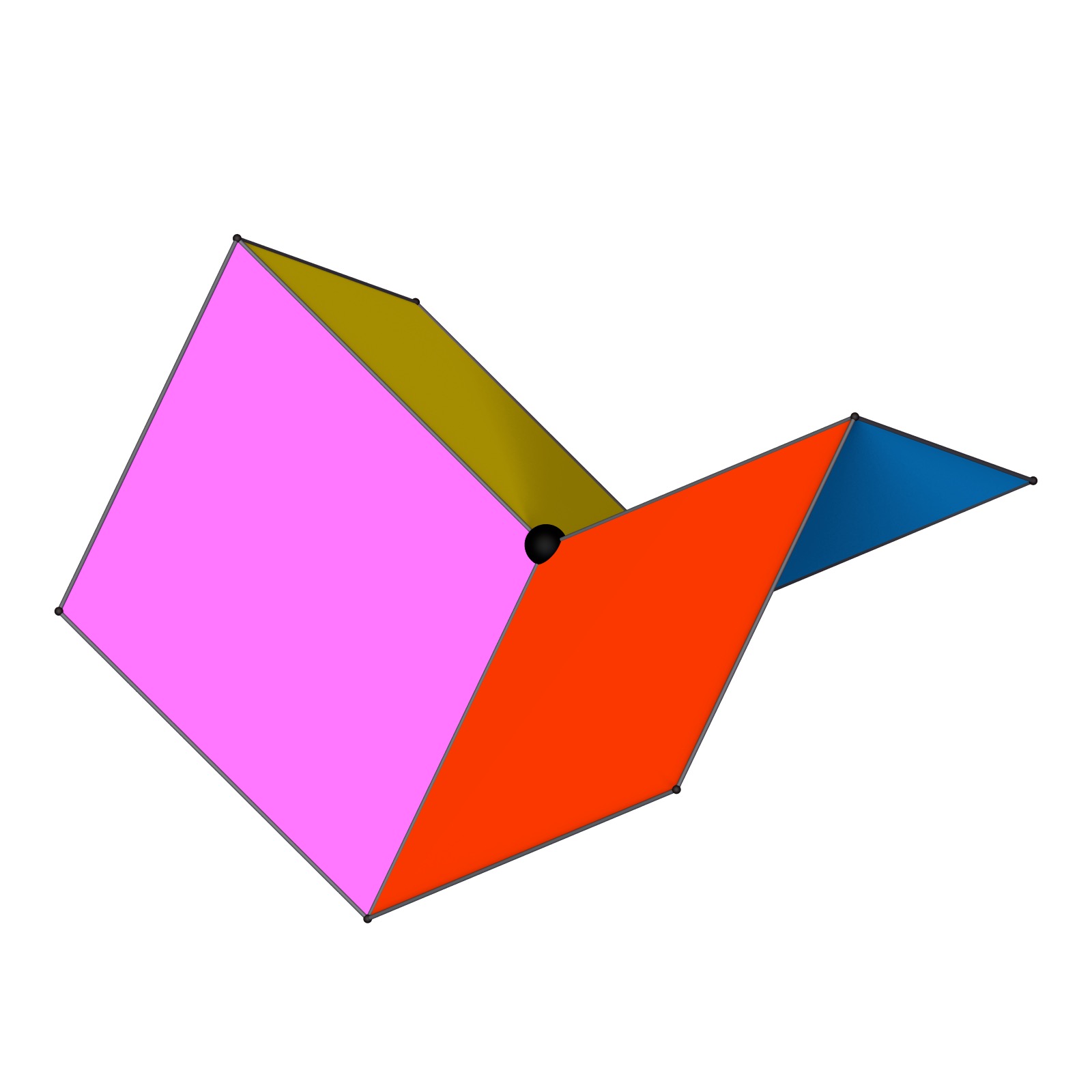}}
\subcaptionbox{Peak $(4,0)$, $\kappa=2\pi-4\gamma$\label{fig:Peak}}
{\includegraphics[width=.32\linewidth]{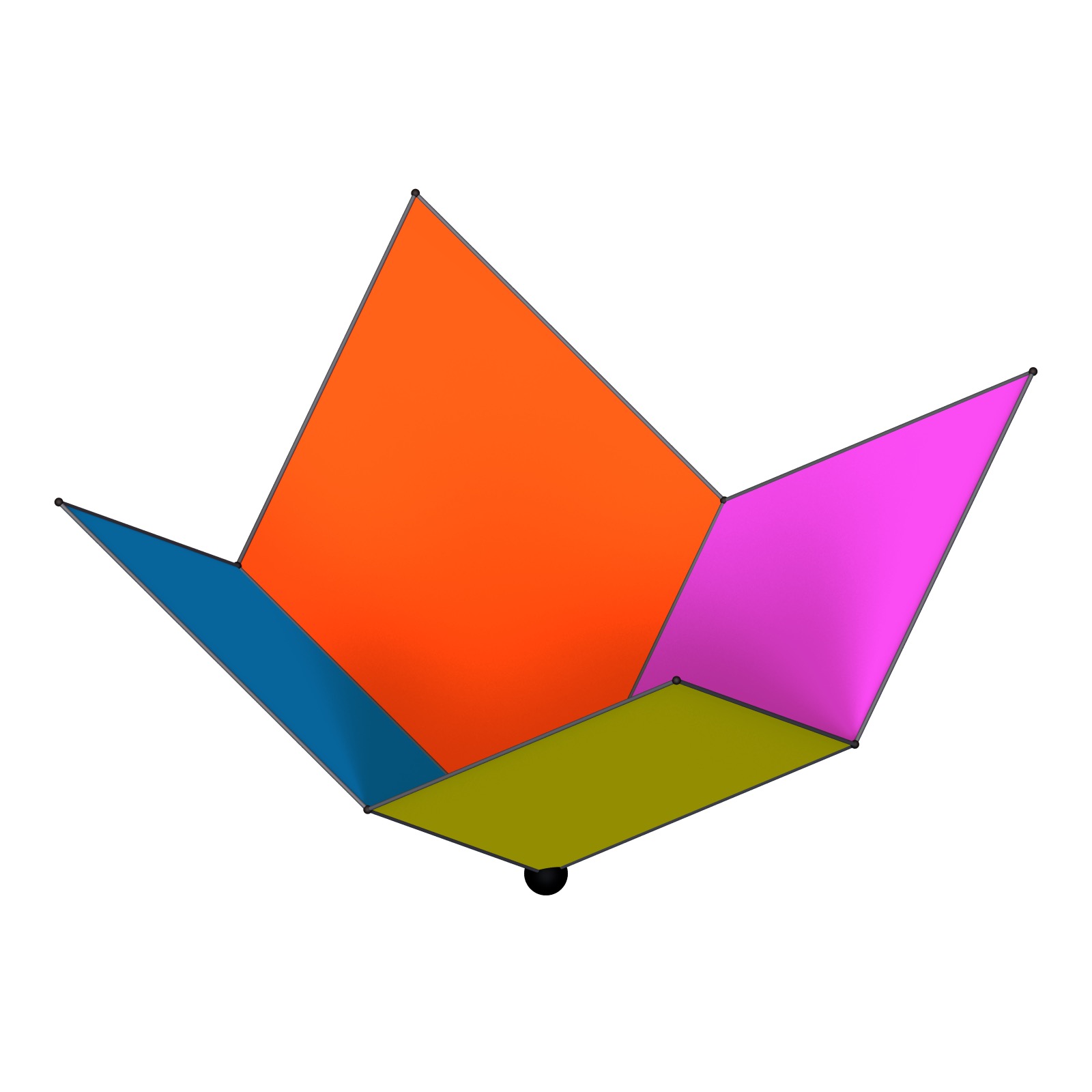}}
\subcaptionbox{Miura $(2,2)$, $\kappa=0$\label{fig:Miura}}
{\includegraphics[width=.32\linewidth]{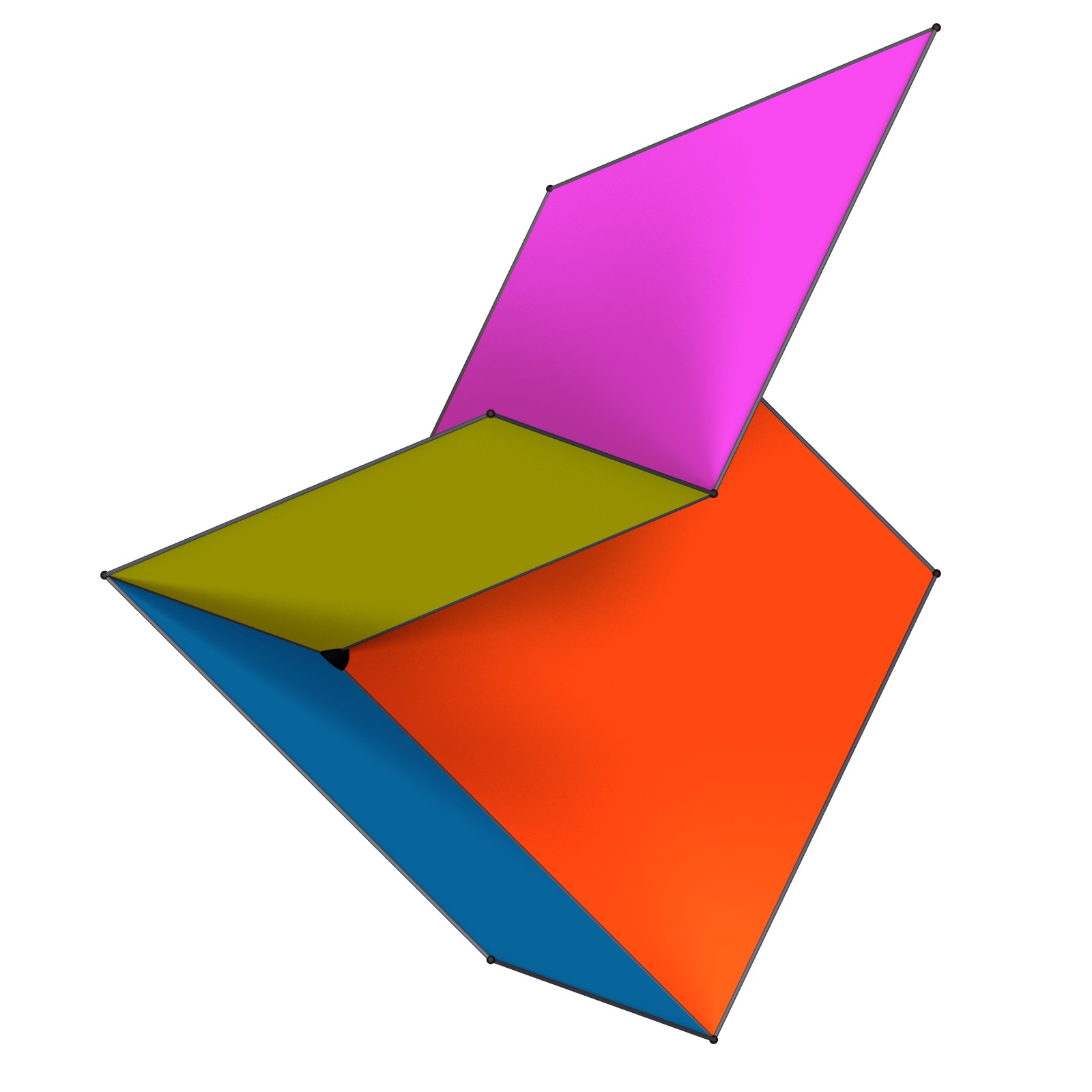}}
\caption{Vertices of valency 4}\label{fig:valency4}
\end{figure}

Next, there are five possible vertex types with valency 6:
    
\begin{figure}[H]
\centering
\subcaptionbox{Crown  $(4,2)$, $\kappa=-2\gamma$  \label{fig:Crown}}
{\includegraphics[width=.32\linewidth]{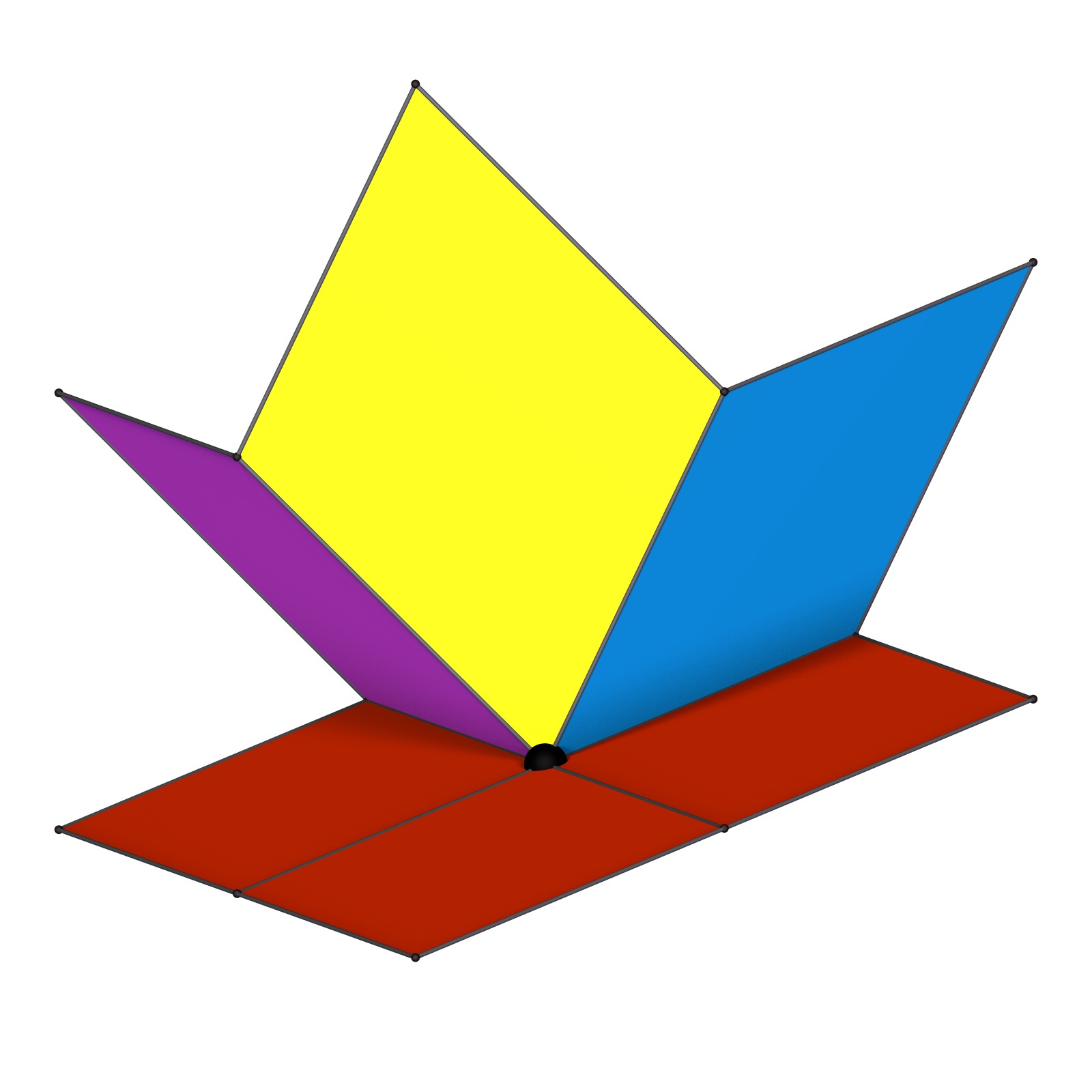}}
\subcaptionbox{Acute X $(4,2)$, $\kappa=-2\gamma$\label{fig:AcuteX}}
{\includegraphics[width=.32\linewidth]{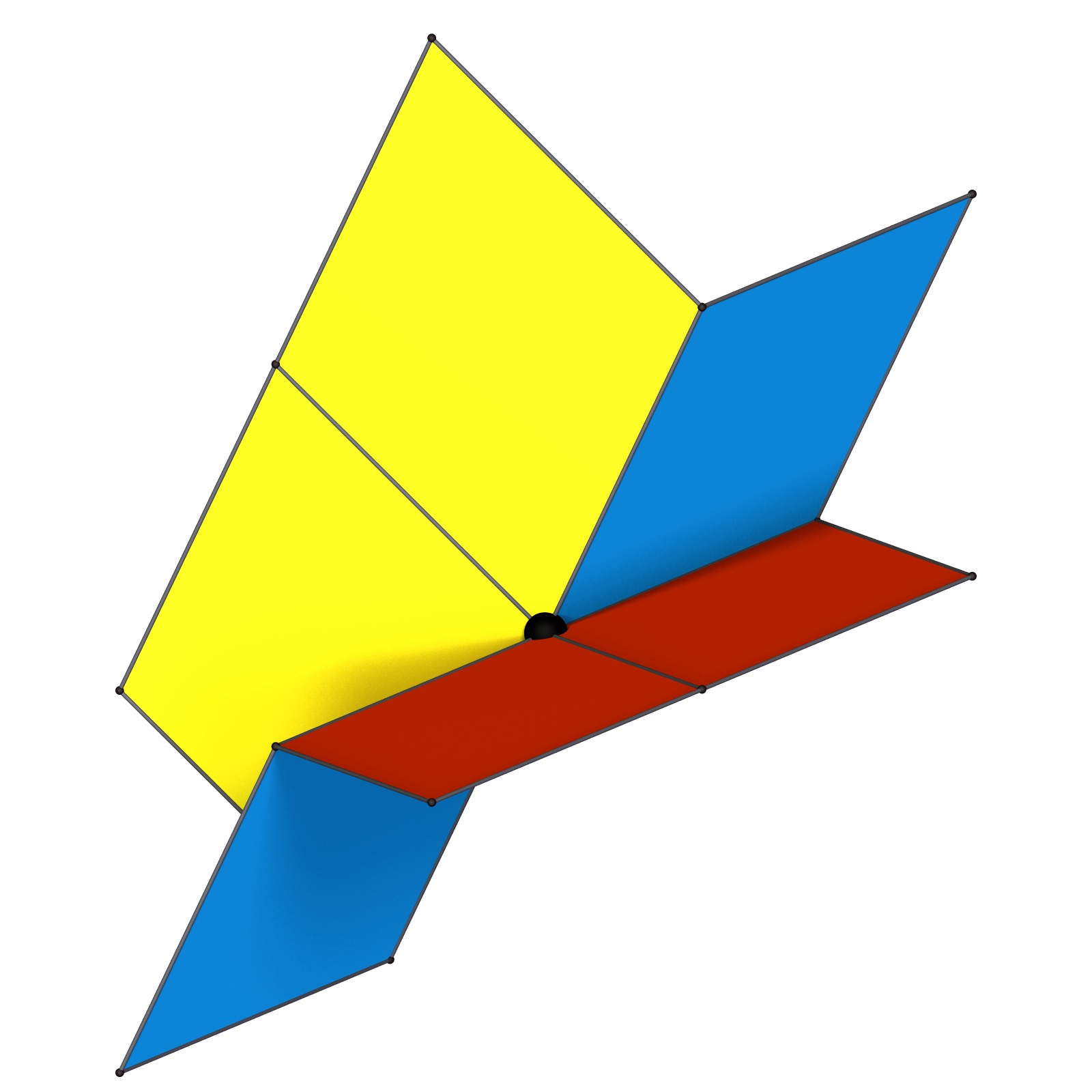}}
\subcaptionbox{Obtuse L $(4,2)$, $\kappa=-2\gamma$\label{fig:ObtuseL}}
{\includegraphics[width=.32\linewidth]{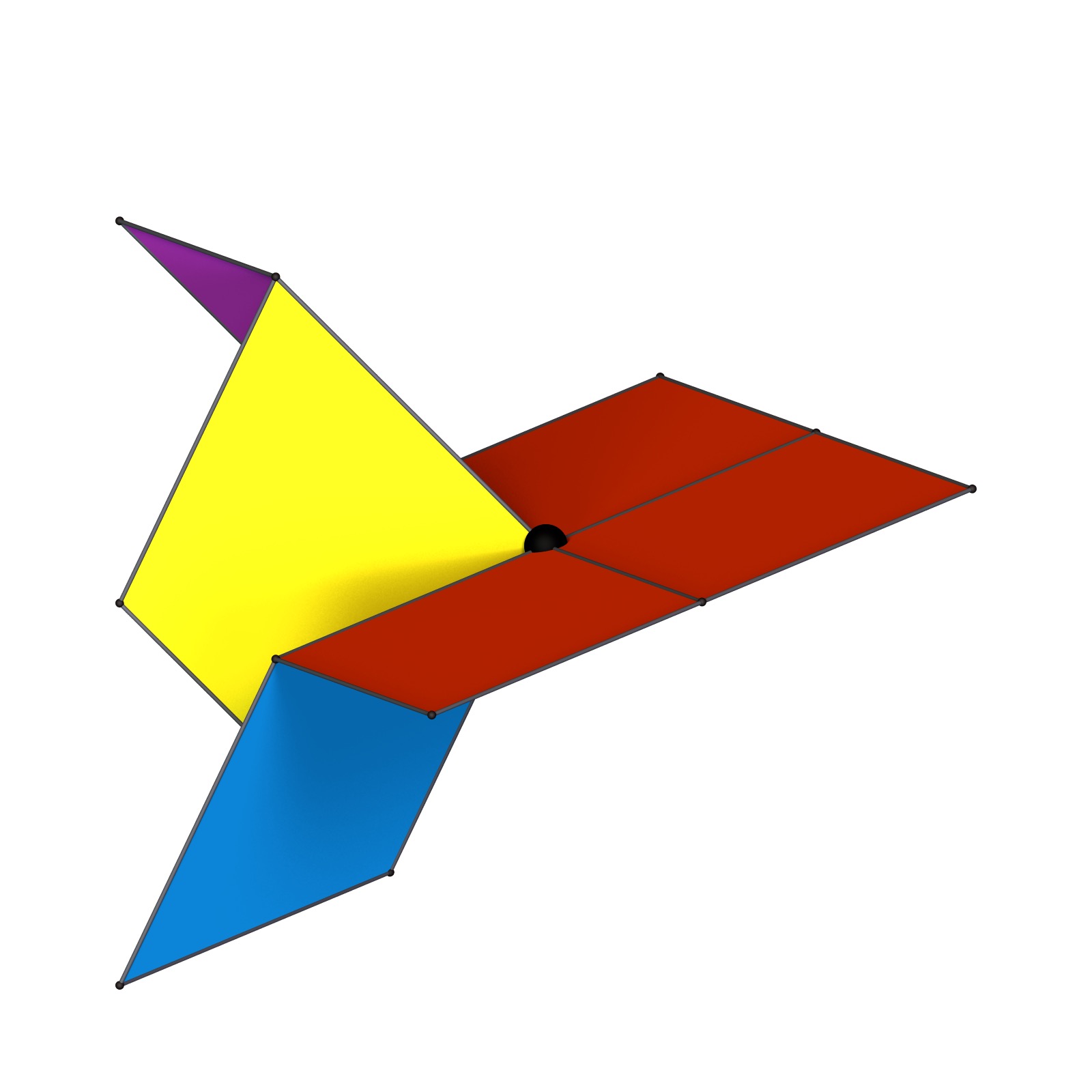}}
\subcaptionbox{Broken Crown $(4,2)$, $\kappa=-2\gamma$\label{fig:BrokenCrown}}
{\includegraphics[width=.32\linewidth]{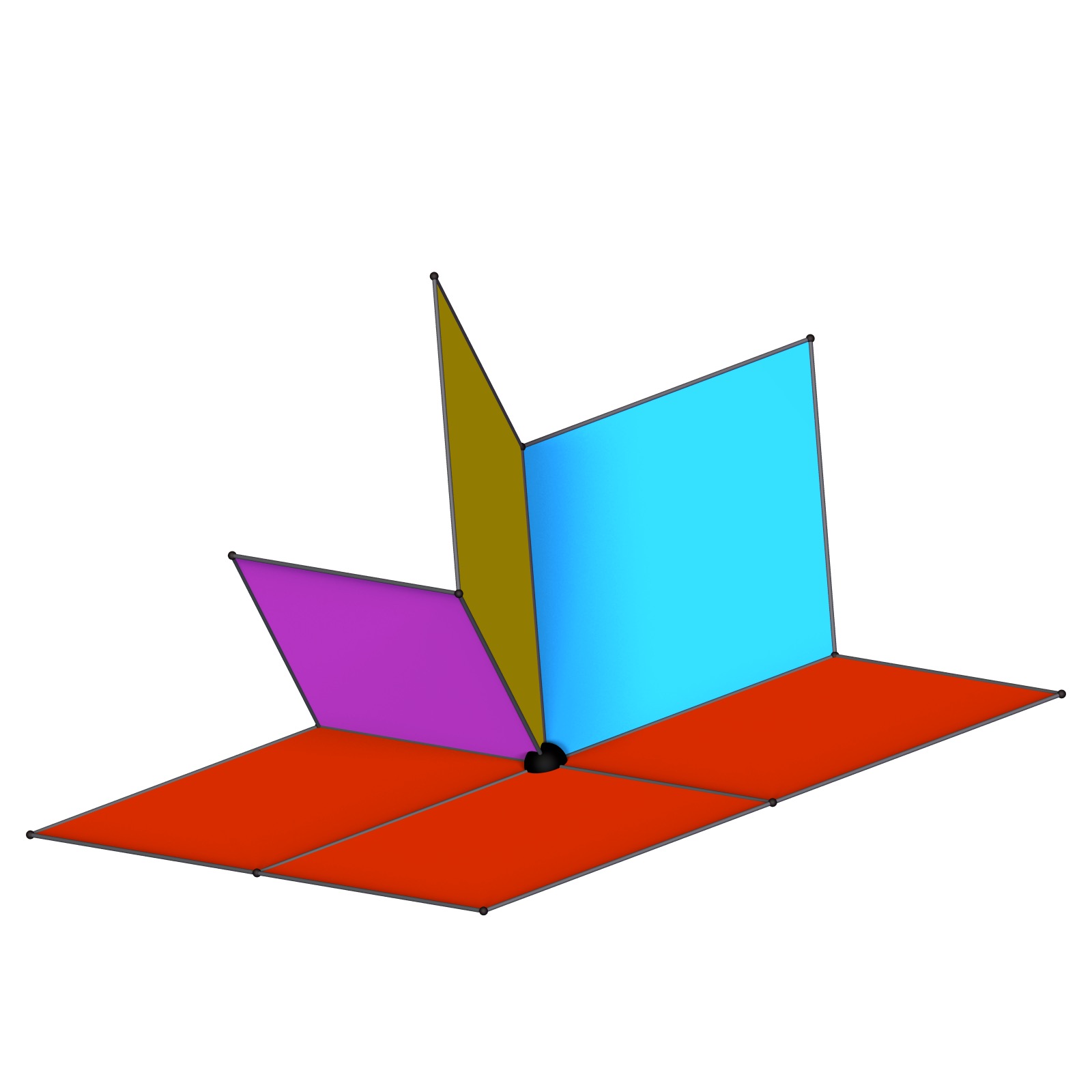}}
\subcaptionbox{Obtuse X $(2,4)$, $\kappa=2\gamma-2\pi$\label{fig:ObtuseX}}
{\includegraphics[width=.32\linewidth]{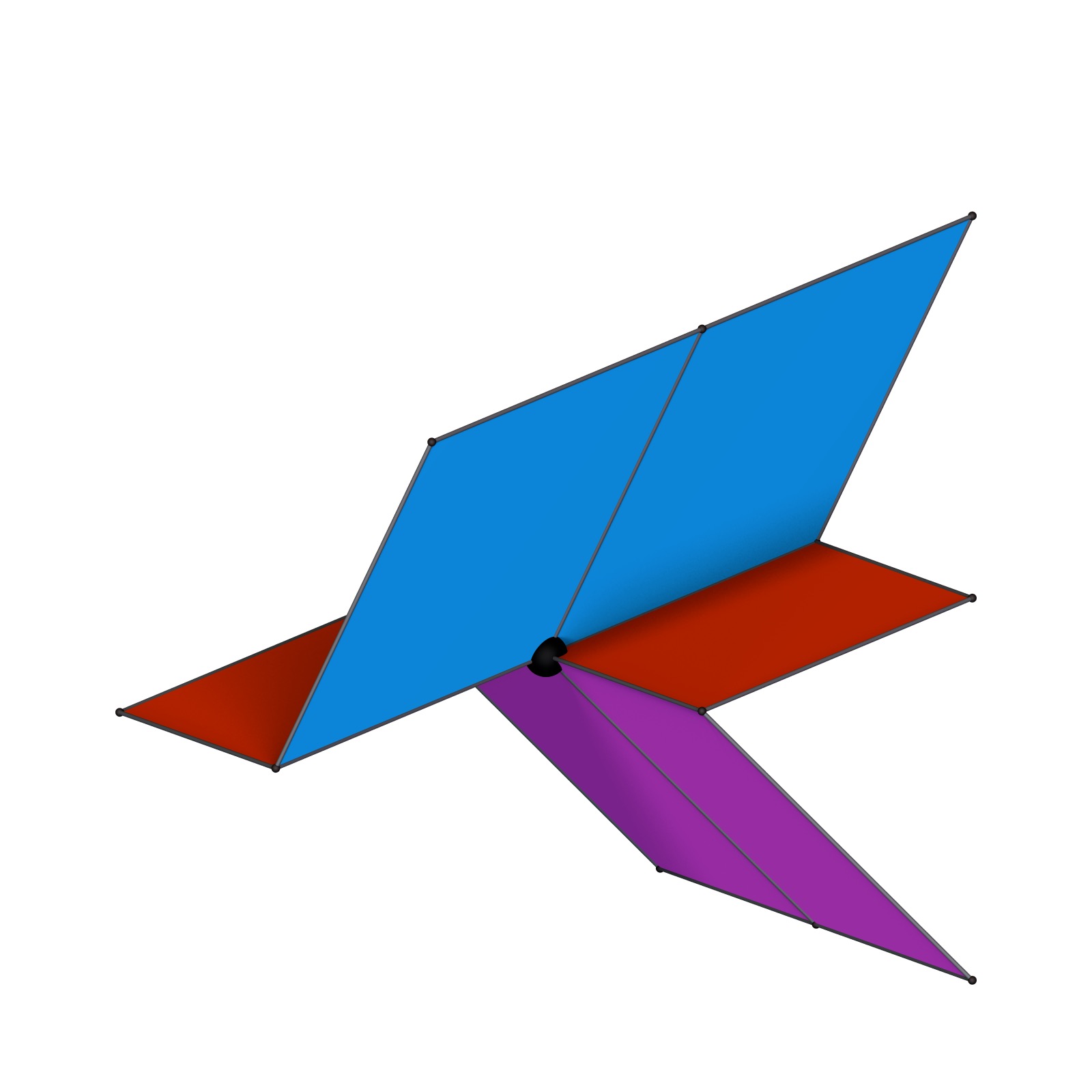}}
\end{figure}

 Finally, there are three possible vertex types with valency 8:
 
\begin{figure}[H]
\centering
\subcaptionbox{Double X  $(6,2)$, $\kappa=-4\gamma$  \label{fig:X}}
{\includegraphics[width=.32\linewidth]{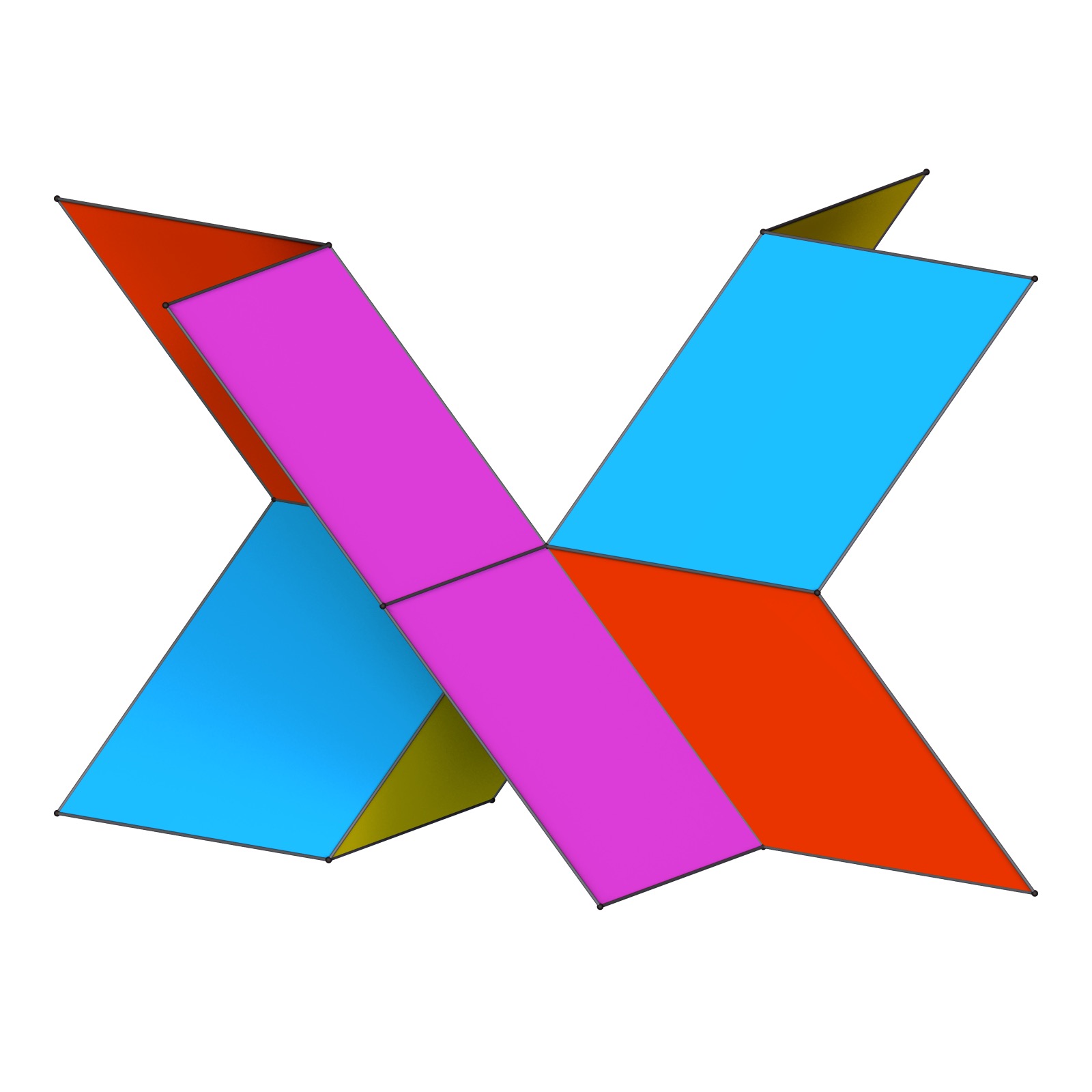}}
\subcaptionbox{Star $(6,2)$, $\kappa=-4\gamma$\label{fig:Star}}
{\includegraphics[width=.32\linewidth]{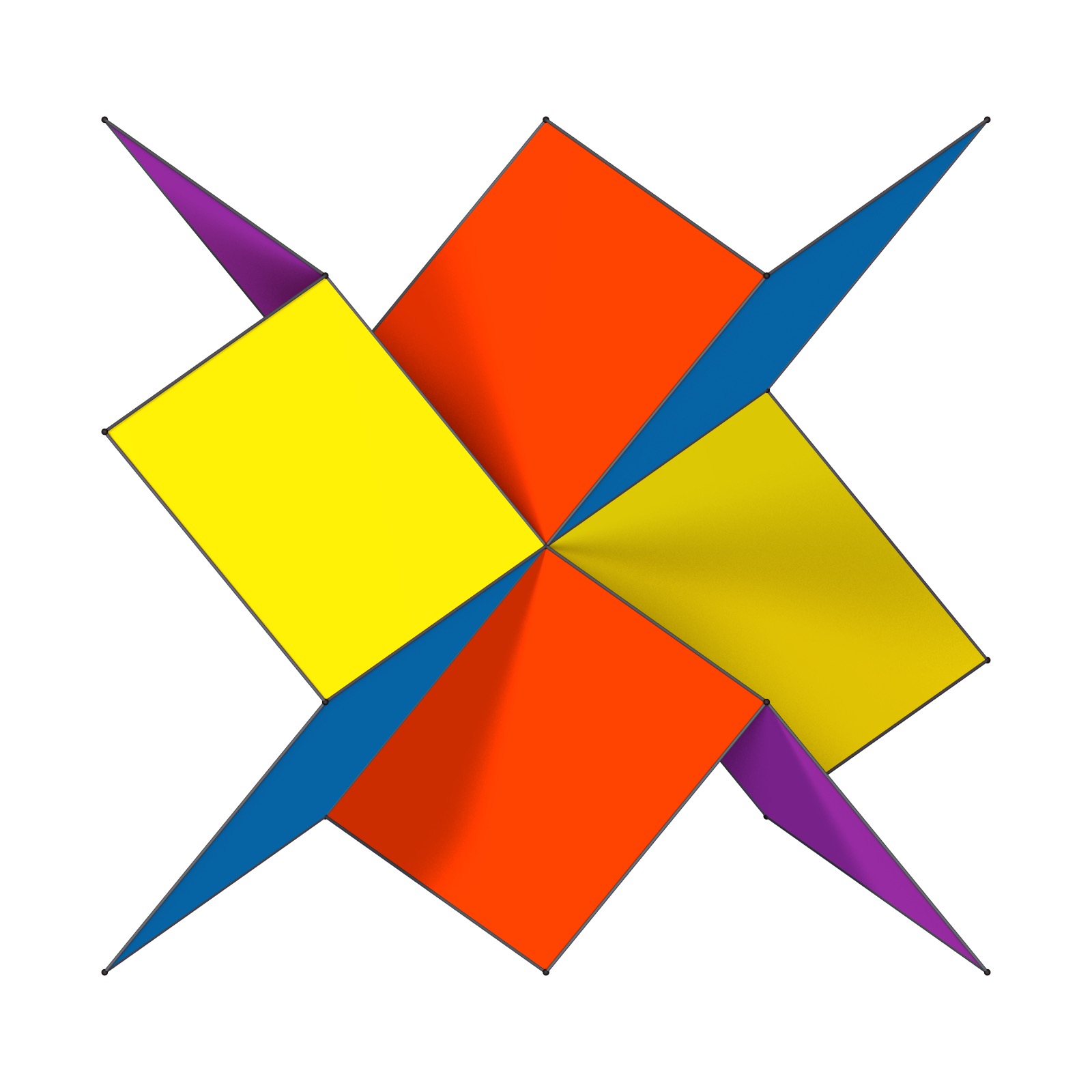}}
\subcaptionbox{Double L $(4,4)$, $\kappa=-2\pi$\label{fig:LL}}
{\includegraphics[width=.32\linewidth]{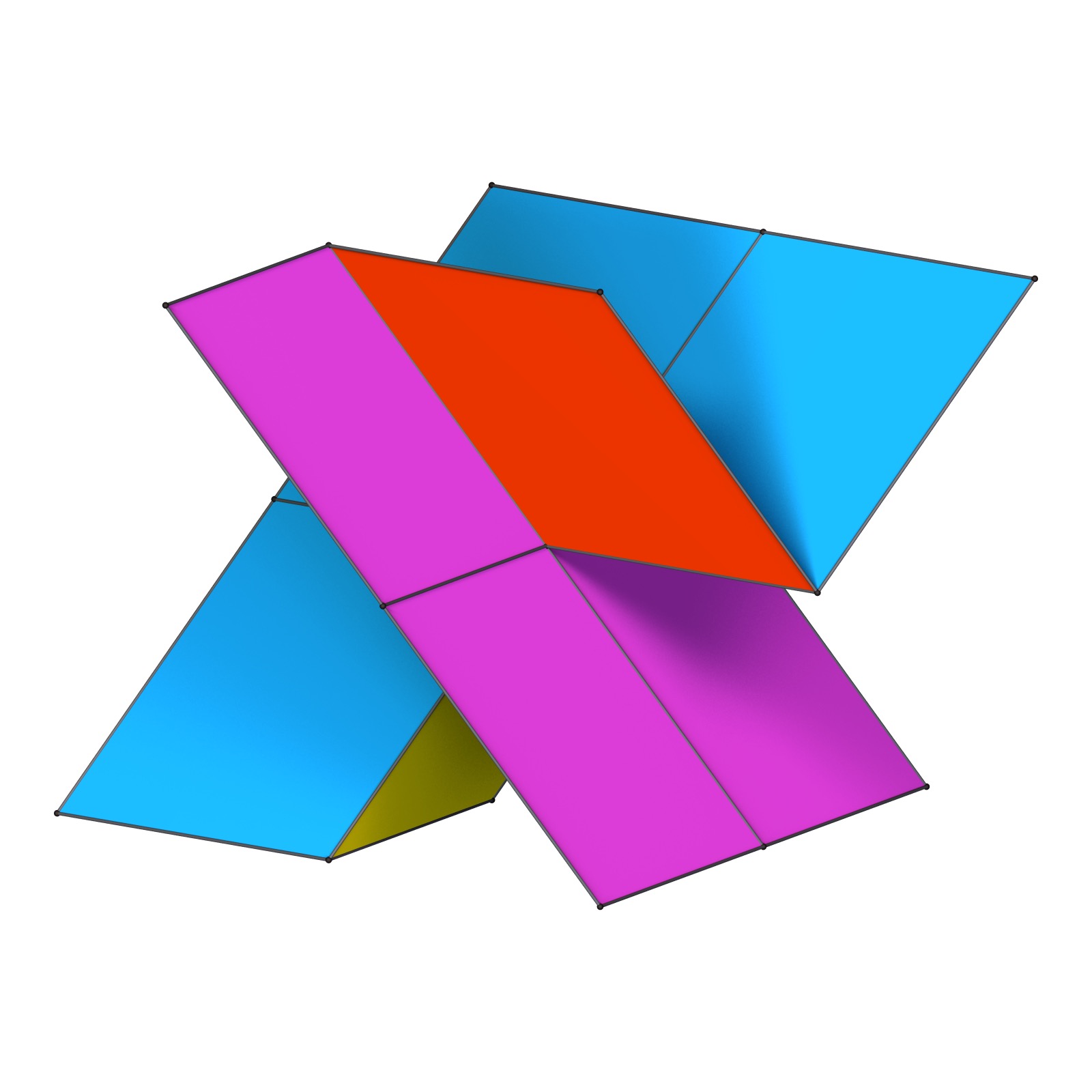}}
\end{figure}

We have found that any of these vertex types can be used in an infinite doubly or triply periodic $\Sigma$-polyhedron.
 
\section{Simple Doubly Periodic Examples}
\label{sec:simple}

In this section, we review two  well known doubly periodic examples. The first is the {\em Eggbox pattern} \cite{alfred1968}.

\begin{figure}[H]
      \centering
      \subcaptionbox{$\alpha=25^\circ$}
        {\includegraphics[width=0.3\textwidth]{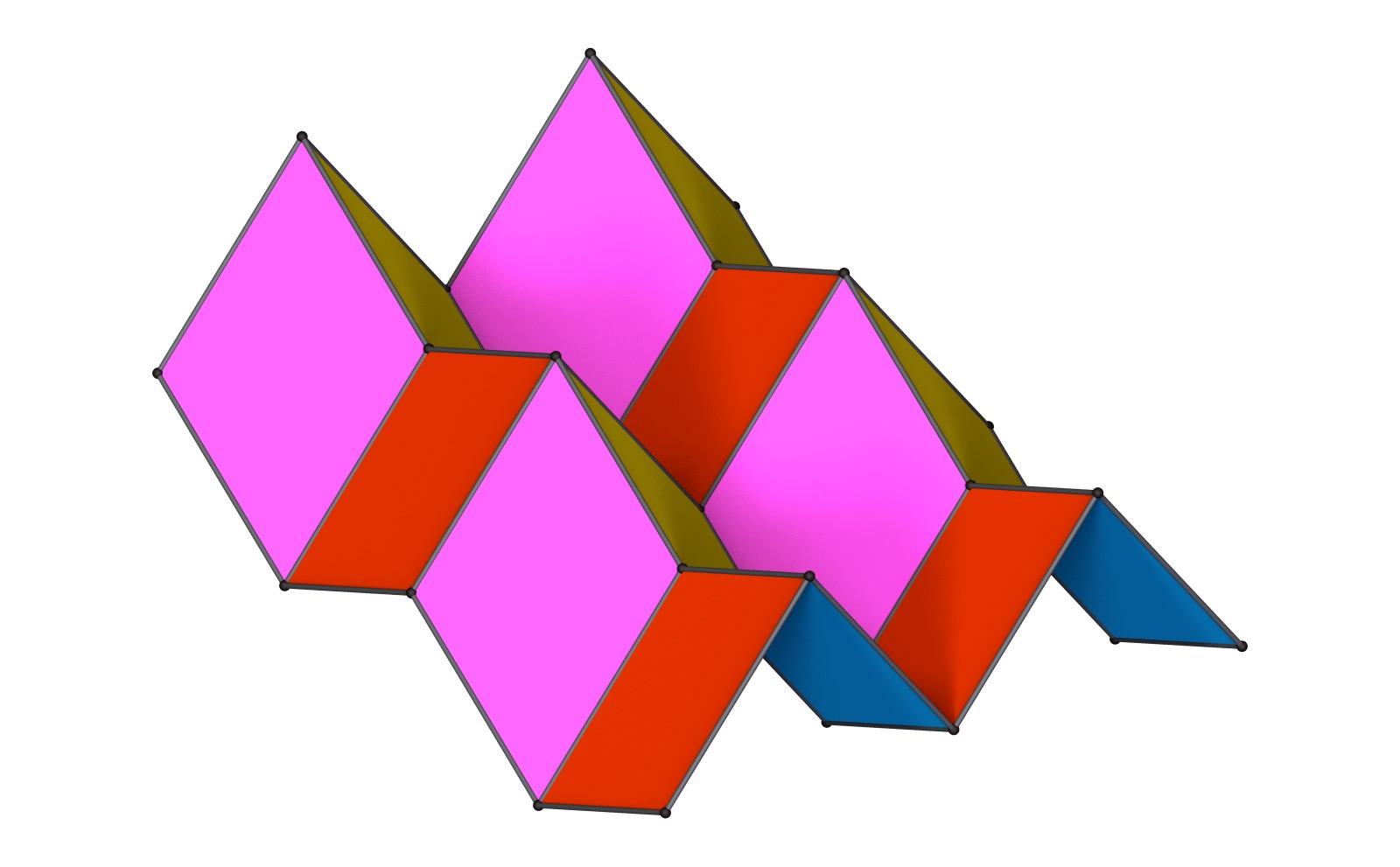}}
      \subcaptionbox{$\alpha=35^\circ$}
        {\includegraphics[width=0.3\textwidth]{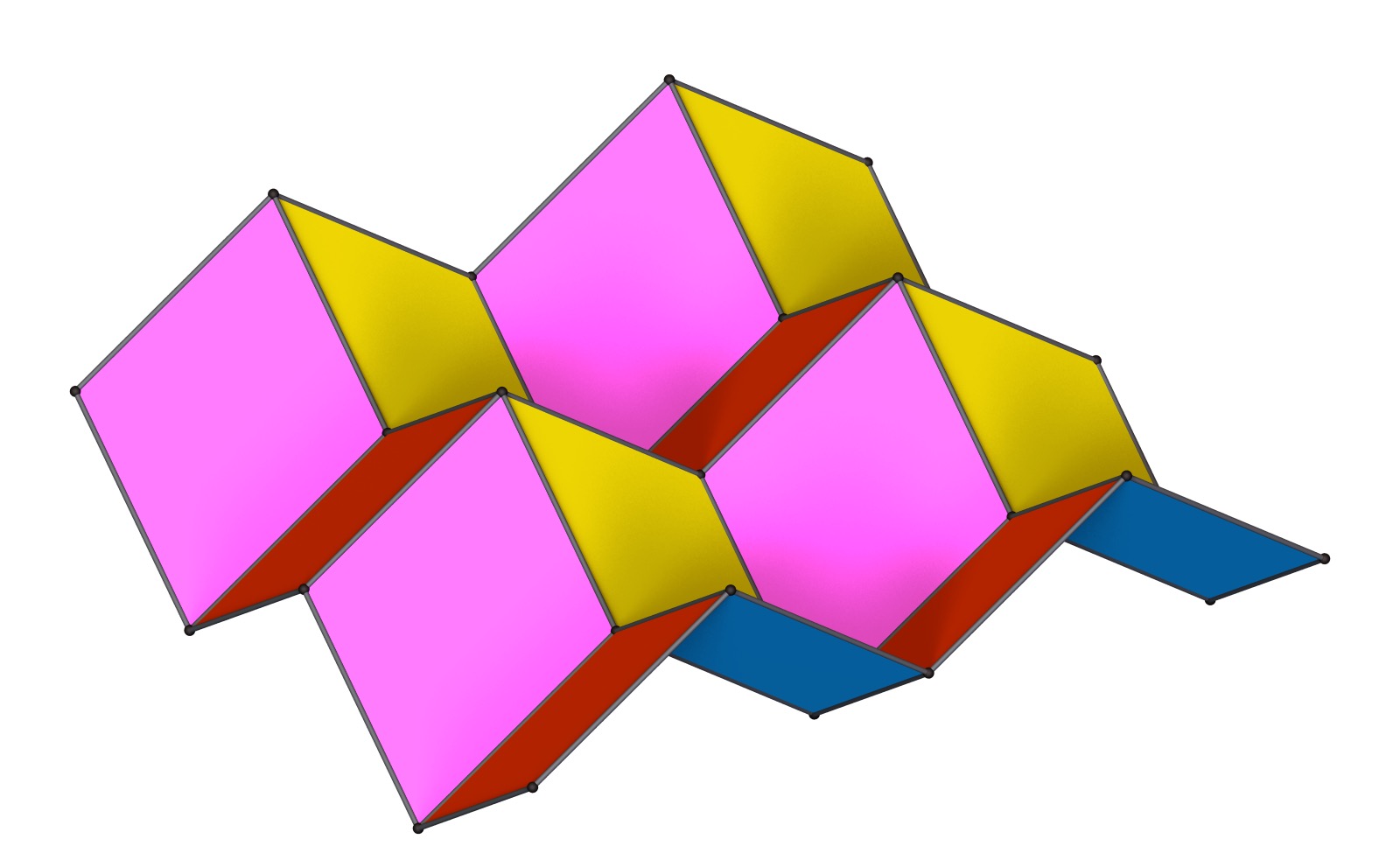}}
        \subcaptionbox{$\alpha=65^\circ$}
        {\includegraphics[width=0.3\textwidth]{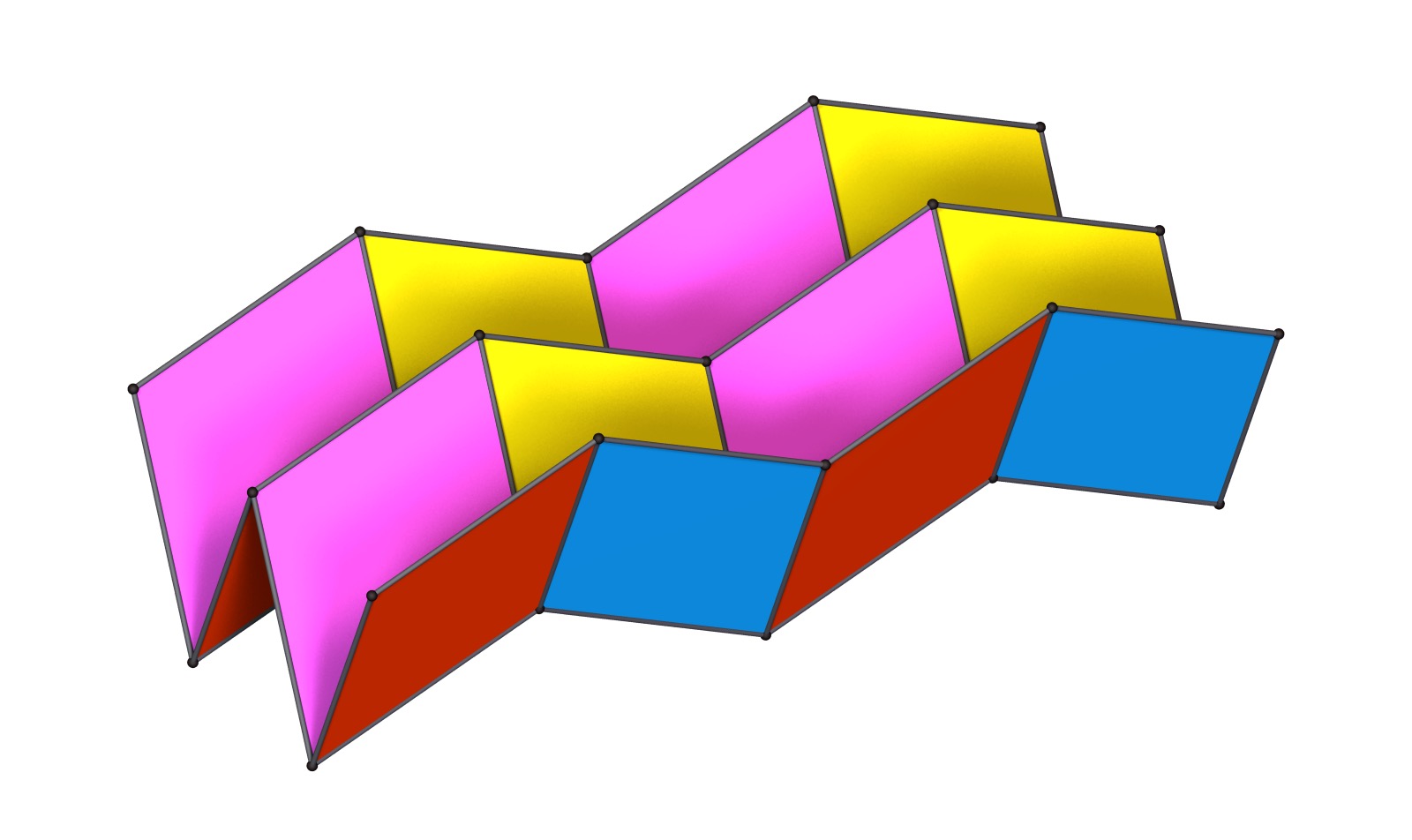}}
      \caption{Three states of the Eggbox pattern}\label{fig:Eggbox}
    \end{figure}
    
In our context a translational fundamental piece is the saddle consisting of  $\Pi_{13}$, $\Pi_{14}$, $\Pi_{23}$ and $\Pi_{24}$, and the translation vectors are
$v_1-v_3$ and $v_3-v_4$.

Another example is the {\em Miura pattern} \cite{miura1969}.

\begin{figure}[H]
      \centering
      \subcaptionbox{$\alpha=25^\circ$}
        {\includegraphics[width=0.3\textwidth]{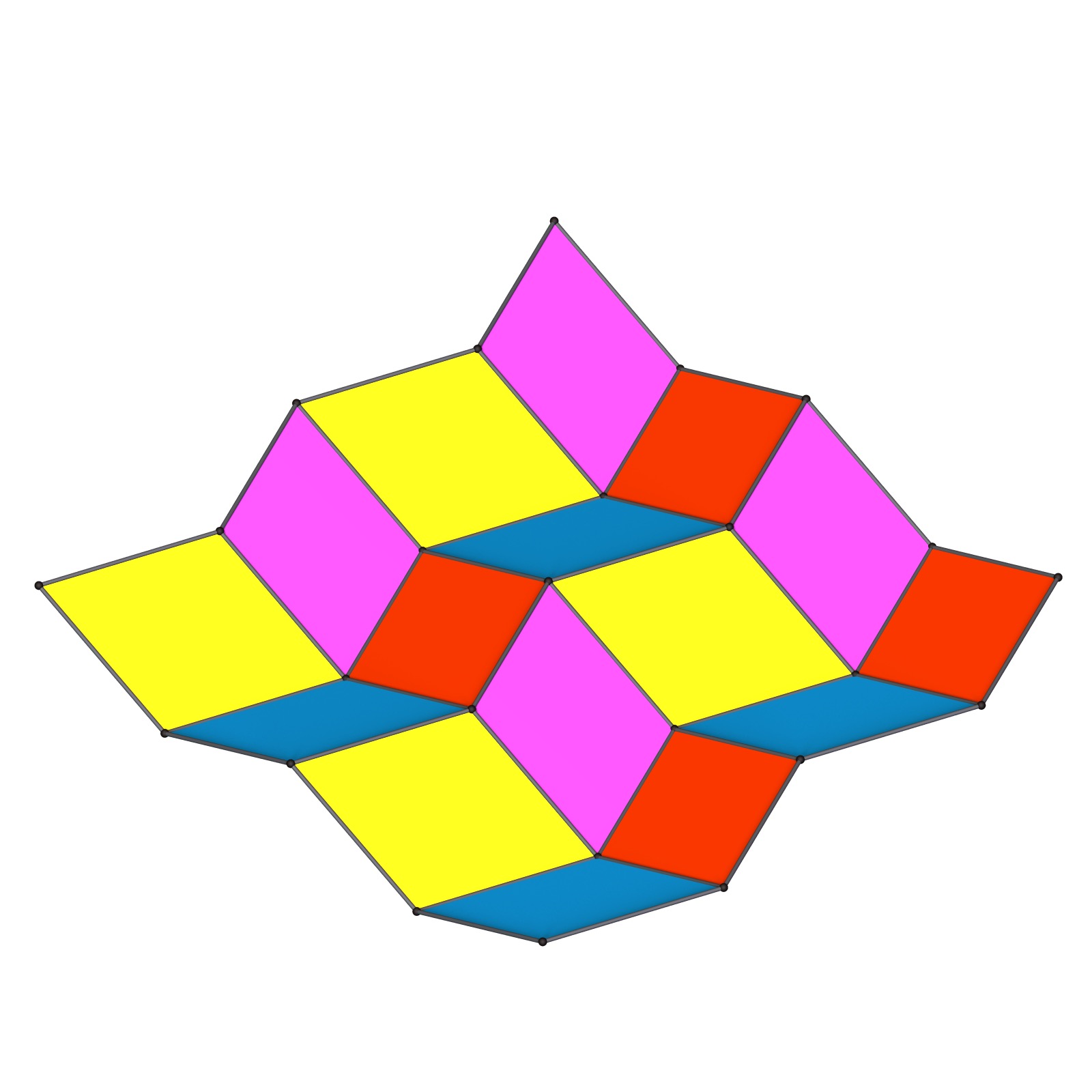}}
      \subcaptionbox{$\alpha=35^\circ$}
        {\includegraphics[width=0.3\textwidth]{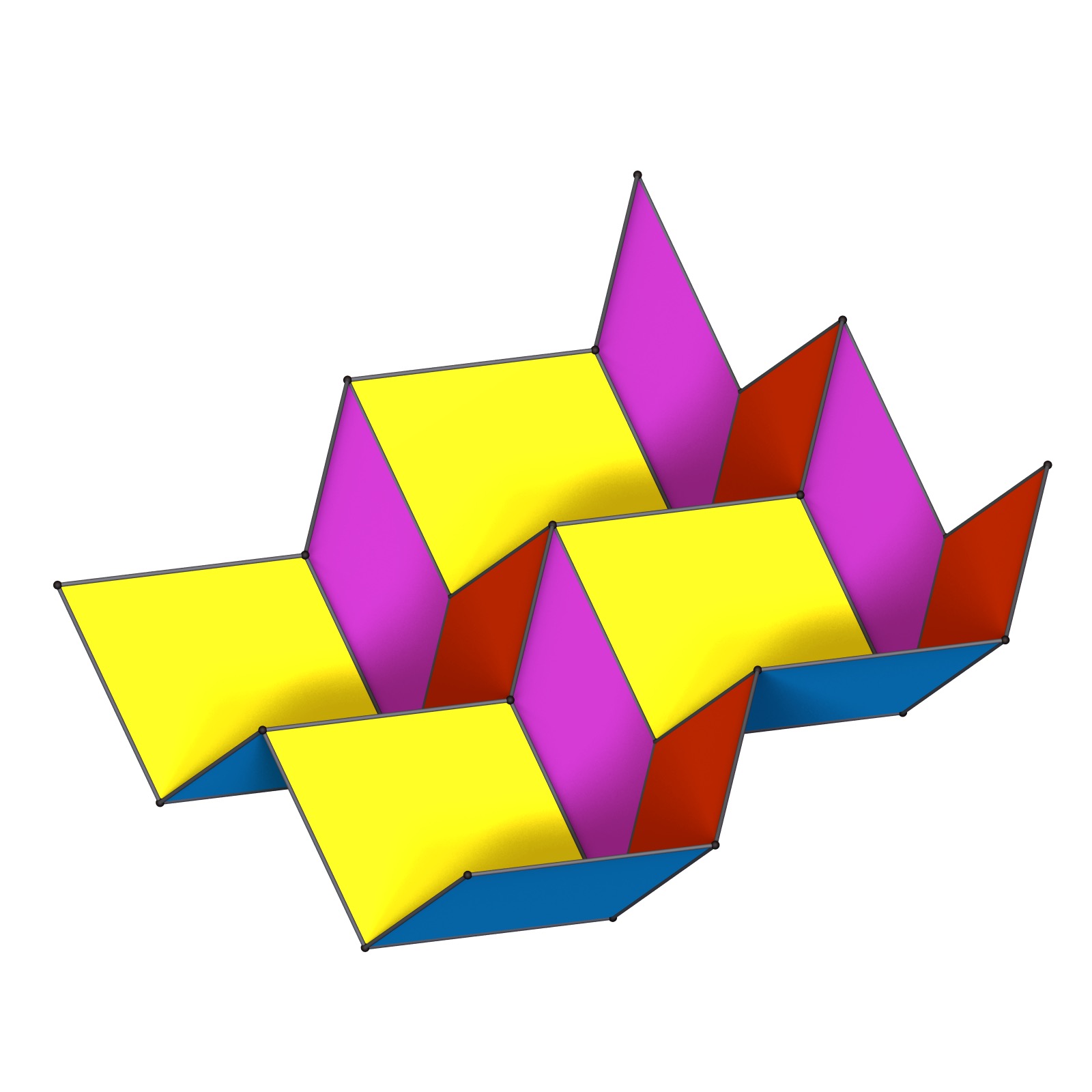}}
        \subcaptionbox{$\alpha=65^\circ$}
        {\includegraphics[width=0.3\textwidth]{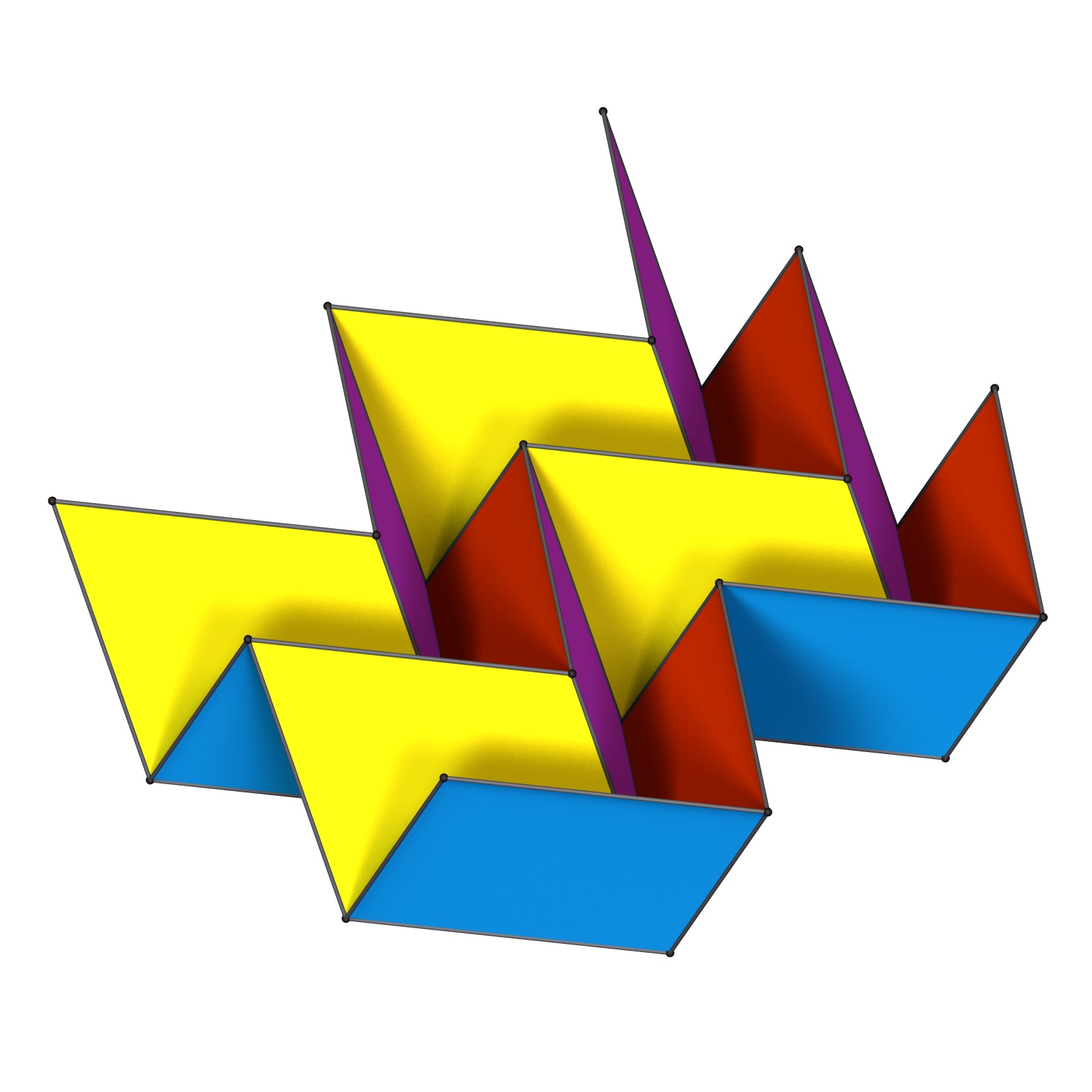}}
      \caption{Three states of the Miura pattern}\label{fig:MiuraPattern}
    \end{figure}
    
  Here, a   translational fundamental piece is also the union of the four admissible facets, but arranged differently, namely as
   $\Pi_{13}$, $\Pi_{14}$, $\Pi_{23}-v_2$ and $\Pi_{24}-v_2$, where we use addition of a vector to denote translation.
   The translational periods are given by $v_1+v_3$ and $v_3-v_4$.

\section{A Bifoldable Fractal}
\label{sec:fractal}

Let $H_i$ be the 2-skeleton of the parallelepiped spanned by the four star vectors except $v_i$, with the forbidden facets removed.
We call $H_i$ a {\em hollowped}. Each $H_i$ is a polyhedral cylinder bounded by four parallelograms. Note that the  four corresponding {\em solid} parallelepipeds tile a solid dodecahedron, which we call the {\em standard dodecahedron}. In the case of the tetrahedral star this is the well-known decomposition of the  rhombic dodecahedron into four congruent rhomboids.

We now use the hollowpeds to inductively construct a fractal bifoldable polyhedron.

The generation 0 fractal $\fF_0$ is the union $H_1\cup H_2\cup H_2\cup H_2$ of all four hollowpeds, see Figure \ref{fig:fractal0}.

\begin{figure}[H]
      \centering
      \subcaptionbox{$\alpha=24^\circ$}
        {\includegraphics[width=0.3\textwidth]{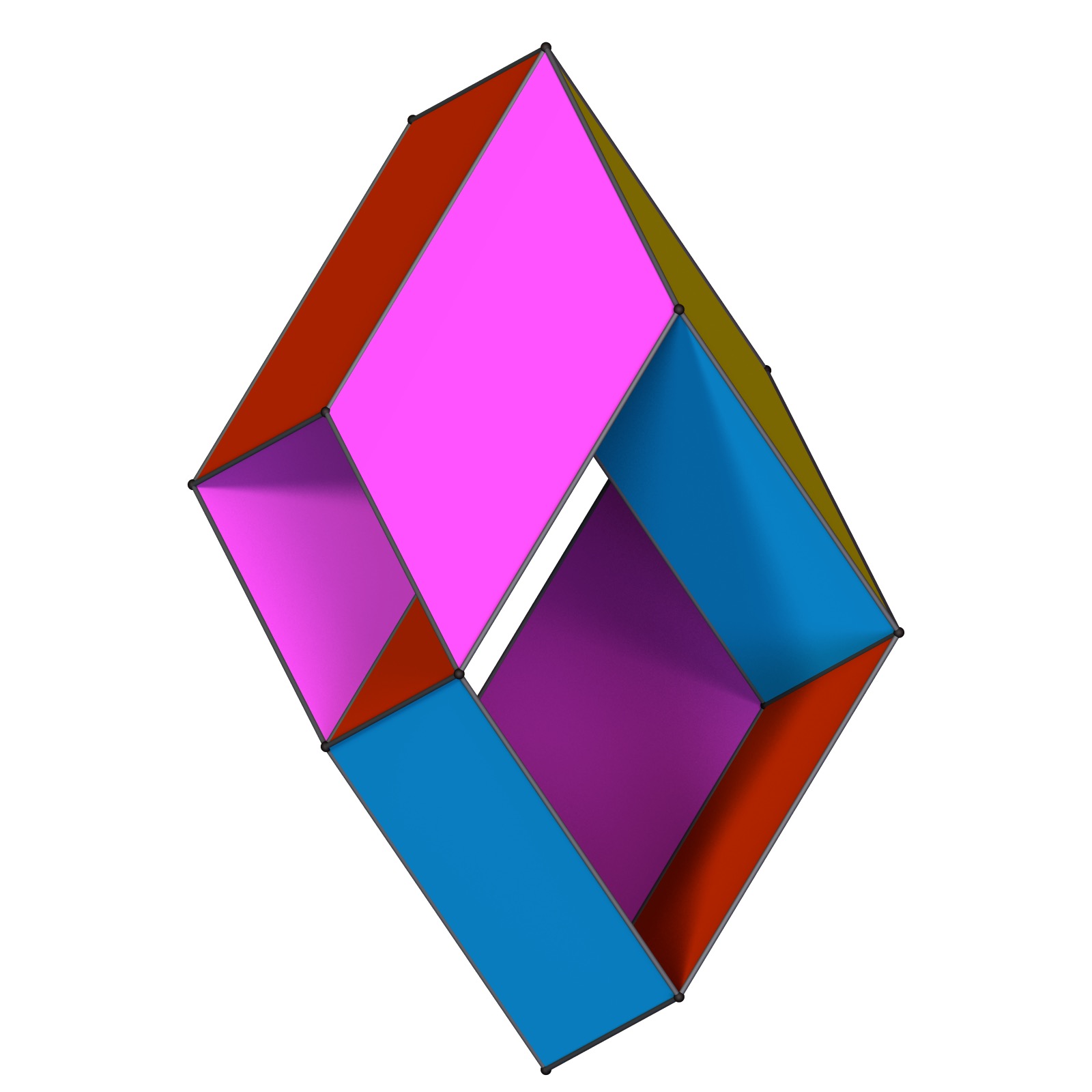}}
      \subcaptionbox{$\alpha=35^\circ$}
        {\includegraphics[width=0.3\textwidth]{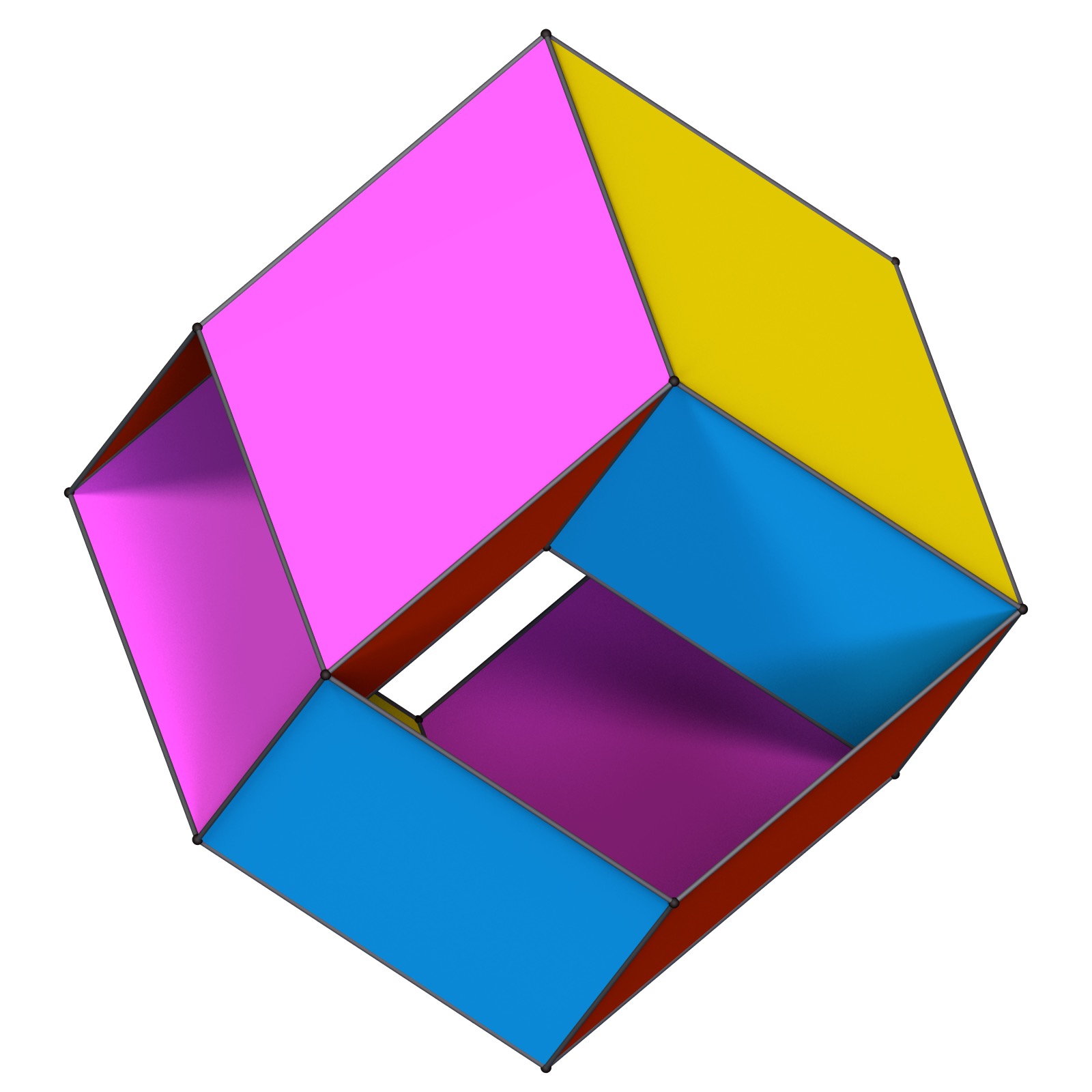}}
        \subcaptionbox{$\alpha=65^\circ$}
        {\includegraphics[width=0.3\textwidth]{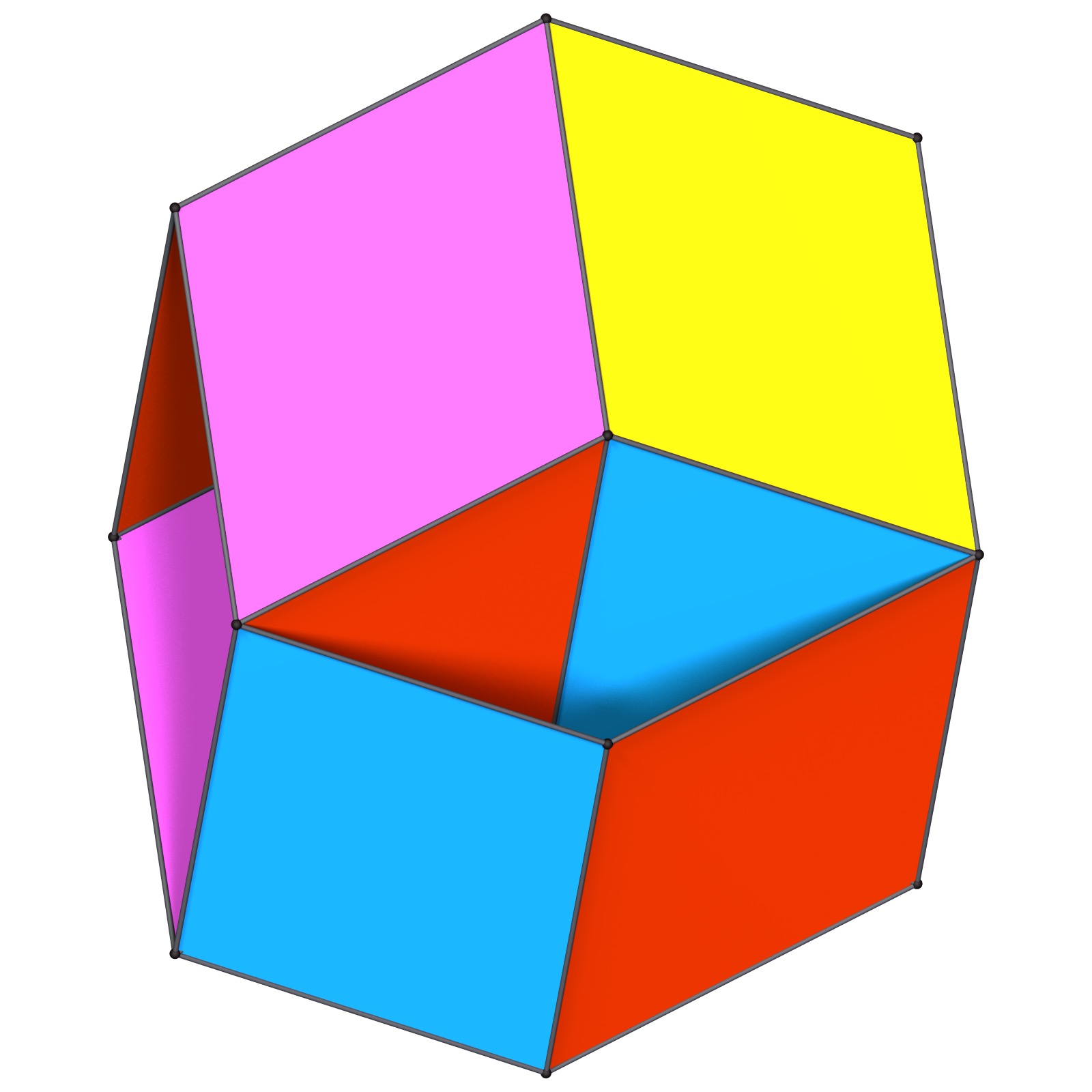}}
      \caption{Three states of the generation 0 fractal $\fF_0$}
      \label{fig:fractal0}
    \end{figure}
    
This shape and its deformability has been discussed before in \cite{cheung2014}, where translational copies of it are used to build a polyhedral complex.

We can think of $\fF_0$  as consisting of a roof, a congruent bottom, and a central saddle. Note that this is a polyhedron with octagonal boundary which is up to similarity identical to the boundary of the saddle. This suggests that whenever we cut a smaller saddle out of a larger saddle, we can fill the hole with a scaled copy of $\fF_0$. 

To make this concrete, we subdivide each parallelogram of $\fF_0$ into four smaller parallelograms, and eliminate the four small parallelograms at the center of the saddle. This turns the saddle into a polyhedral annulus, with two saddle shaped boundary octagons. We can fit a scaled copy of $\fF_0$ into the inner boundary and arrive at the generation 1 fractal $\fF_1$, shown in Figure \ref{fig:Fractal-1}.

\begin{figure}[h]
      \centering
      \subcaptionbox{$\fF_1$ \label{fig:Fractal-1}}
        {\includegraphics[width=0.45\textwidth]{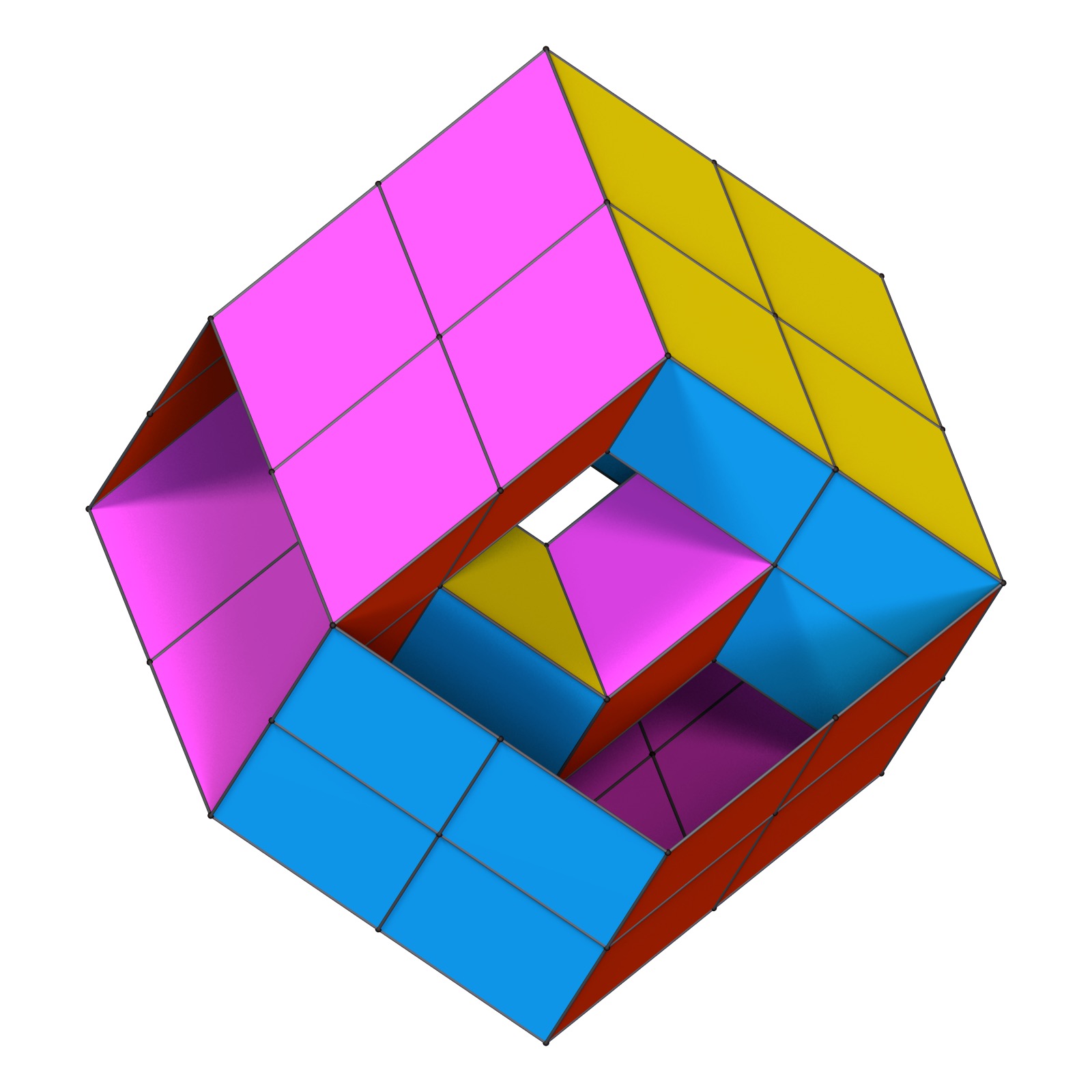}}
      \subcaptionbox{$\fF_2$ \label{fig:Fractal-2}}
        {\includegraphics[width=0.45\textwidth]{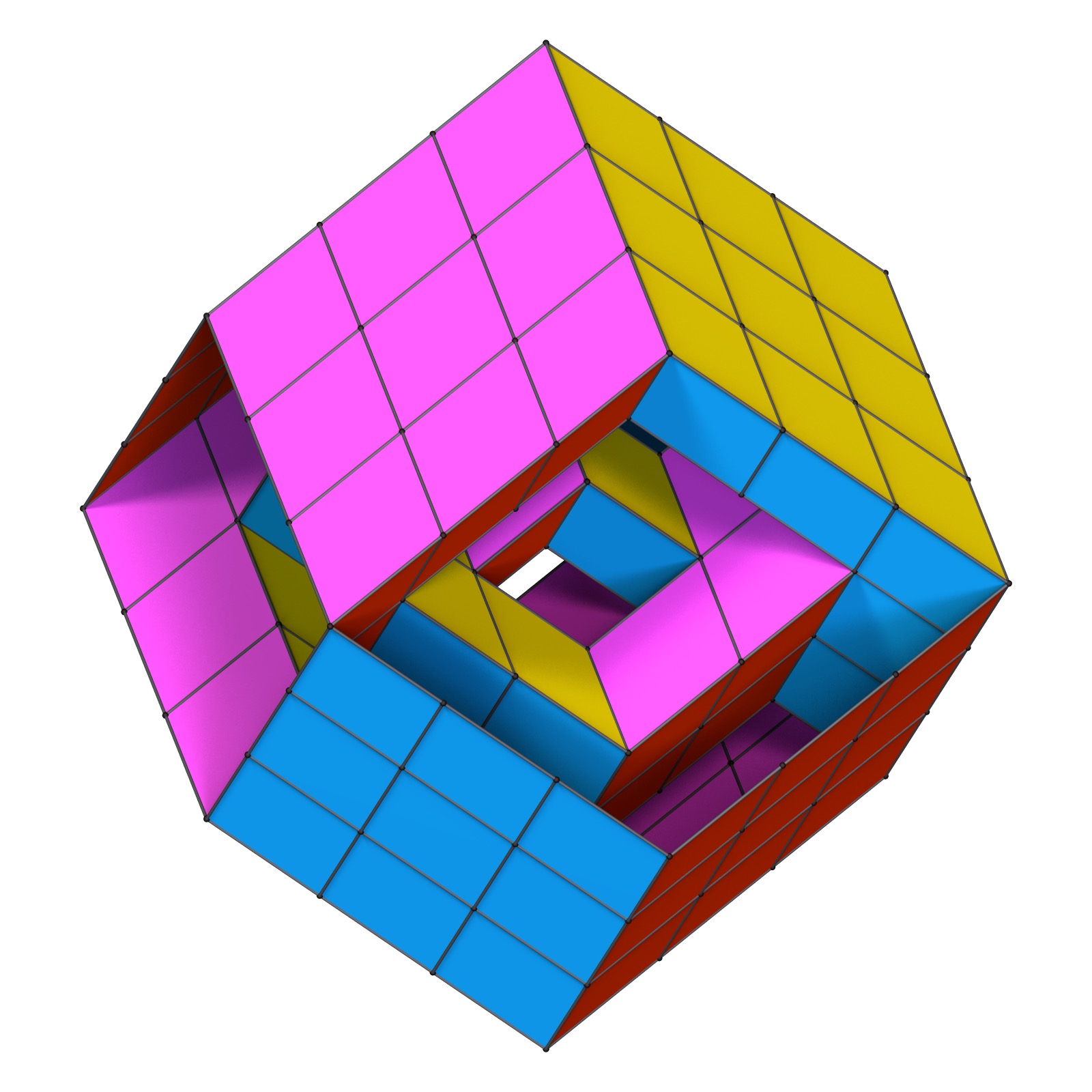}}
      \caption{Fractal generations}
      \label{fig:fractal12}
    \end{figure}

This process can be repeated indefinitely, creating an infinite polyhedron that is the union of polyhedral annuli. By letting the size of these annuli   grow exponentially instead of linearly, one can obtain a  $\Sigma$-polyhedron that is self-similar by a homothety and that has a singular point at the origin.

\begin{remark}
This fractalization procedure can be carried out whenever a bifoldable polyhedral complex contains a vertex of saddle type. It is one of the few examples of local surgery that we found possible with $\Sigma$-polyhedra.
\end{remark}

\section{The Design of Bifoldable Polyhedral Complexes}
\label{sec:design}

To design a $\Sigma$-complex, the following construction method can be used: 

Use a small $\Sigma$-polyhedron like the generation 0 fractal as a template that shows how the facets can be attached. 
Distinguish the  four admissible facets by color,  both for the template and the polyhedron you are building. Thus each color also represents one of the four planes spanned by $v_1\wedge v_3$, $v_1\wedge v_4$, $v_2\wedge v_3$, and $v_2\wedge v_4$. Note that the facets parallel to $v1 \wedge v2$ and to  $v_3 \wedge v_4$ are the  forbidden facets. 

Then begin with any of the fourteen vertex types from section \ref{sec:vertices} and build the vertex figure by placing a colored facet at an existing edge so that it is parallel to the facet of the same color in your  template. 

Notice that at any given edge of a facet, there are at only four facets that can use this edge, of only two different colors. Thus to extend a $\Sigma$-polyhedron across one edge of a specific facet, there are at most three different ways to do so. Some of these might be impossible as they would cause the added facet to intersect with other parts of the polyhedron, as shown in Figure \ref{fig:FourWays}.

This procedure will ensure that the polyhedral complex will be built using only star vectors as edges and using only the admissible facets.

Whenever your polyhedron or polyhedral complex closes up (i.e. has vertices are edges meet), you need to test its flexibility by verifying the homological condition. This arises when adding a facet creates a new closed loop of edges, and you need to verify that this loop uses each type of edge as often forward as backward, when followed once around.

We also mention a second way to design $\Sigma$-complexes. You can begin with a tiling of a portion of space by the four parallelepipeds $R_{ijk}$, where $i,j,k$ designate three different numbers from $\{1,2,3,4\}$, and $R_{ijk}$ is spanned by  $v_{i_1}$, $v_{i_2}$ and $v_{i_3}$. Then remove from all parallelepipeds the forbidden facets to obtain the four hollowpeds. The result is always a $\Sigma$-complex. By removing more facets where more than two facets meet at the same edge, you can then try to make the $\Sigma$-complex a $\Sigma$-polyhedron.

A $\Sigma$-complex becomes periodic when it becomes invariant under translations  in one ore more dimensions. We will present examples of bifoldable  periodic $\Sigma$-complexes in the following sections.

\section{The Origami Tube and Weave}
\label{sec:weave}

In this section, we create a doubly periodic version of the Origami tube.
The Origami tube \cite{miura1969,filipov2015} is the union of $\Pi_1$ and $\Pi_3$ and invariant under translations by $v_3-v_4$, see Figure \ref{fig:MiuraTube}.

\begin{figure}[H]
   \centering
   \includegraphics[width=4in]{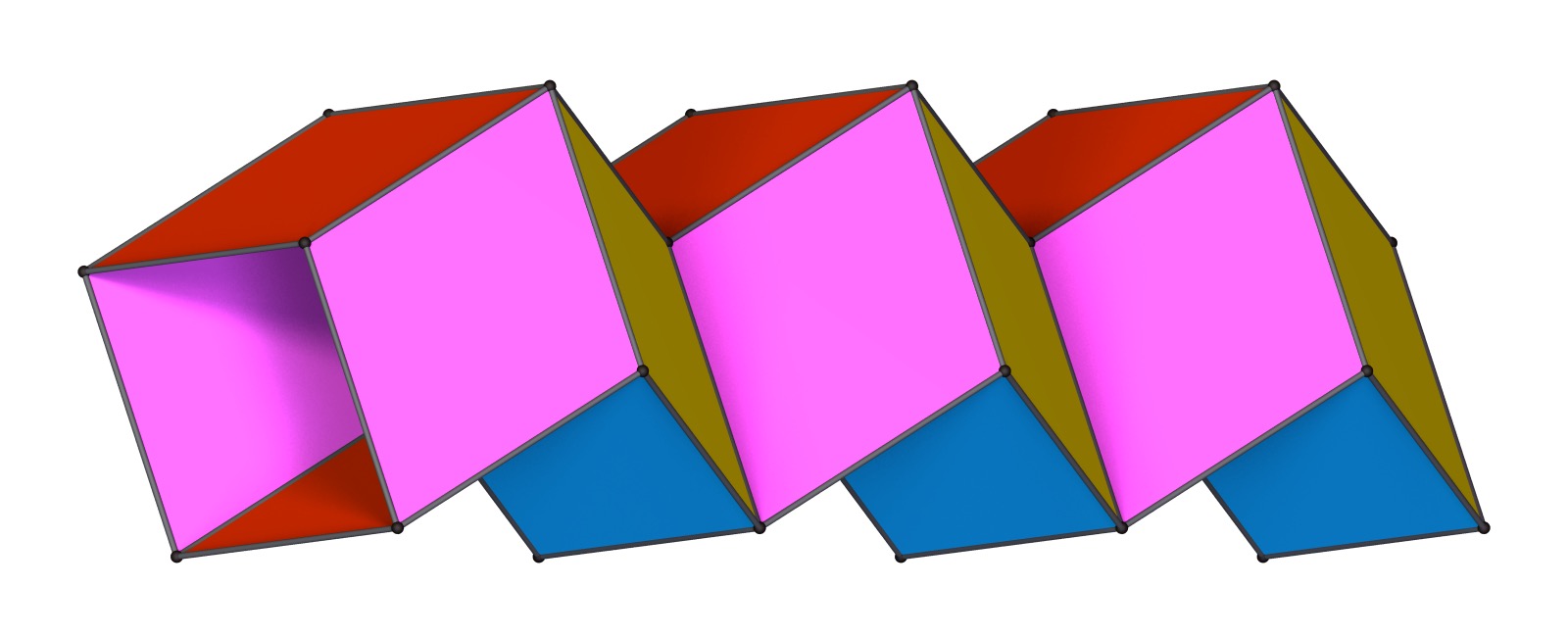}
   \caption{The Origami Tube}
   \label{fig:MiuraTube}
\end{figure}

Using our framework, it is easy to create a doubly periodic pattern of interwoven Origami tubes. To see this, we begin by arranging three facets of type $\Pi_{14}$ into an  L-shape, and we do likewise  with  three facets of type $\Pi_{23}$. These two $L$-shapes are then joined using two single facets $\Pi_{13}$ and $\Pi_{24}$ to create what we called the {\em Double L}, an arrangement of eight facets around a single vertex, see Figure \ref{fig:LL}. 

Four of the Double L shapes can then be combined (using mirrored copies) to create a translational fundamental piece of the {\em Miura Weave}.

\begin{figure}[H]
   \centering
   \includegraphics[width=3in]{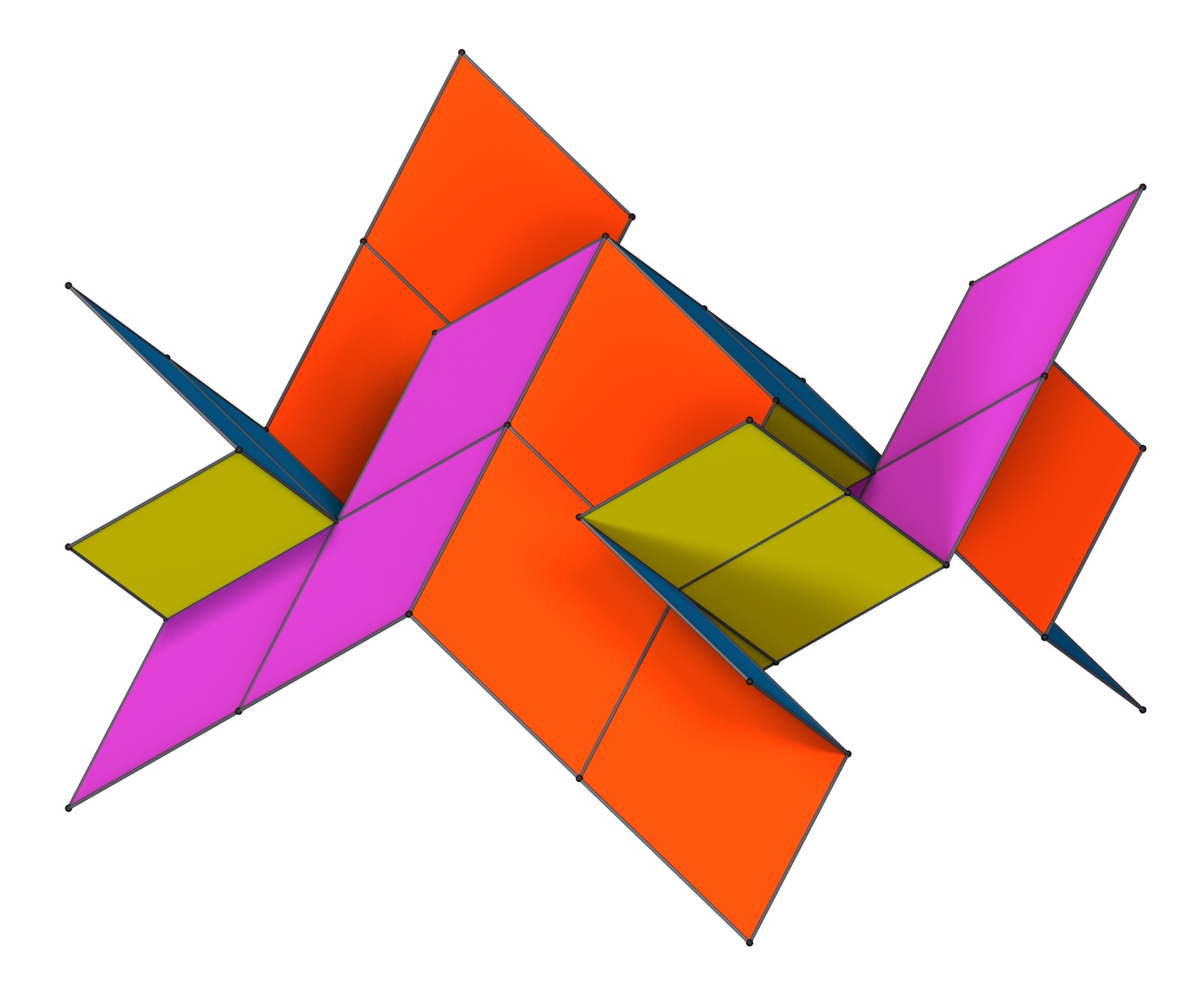}
   \caption{A translational fundamental piece}
   \label{fig:MiuraWeave-2}
\end{figure}

Replicating the fundamental piece using translations by $2(v_3-v_4)$ and $2(v_1-v_2)$ creates a doubly periodic polyhedral carpet without boundary.
Remarkable, the folding happens in the two translational directions, allowing to compress the entire carpet into a thin strip (Figure \ref{fig:MiuraWeave-4}).

\begin{figure}[H]
      \centering
      \subcaptionbox{Symmetric state\label{fig:MiuraWeave-3}}
        {\includegraphics[width=0.49\textwidth]{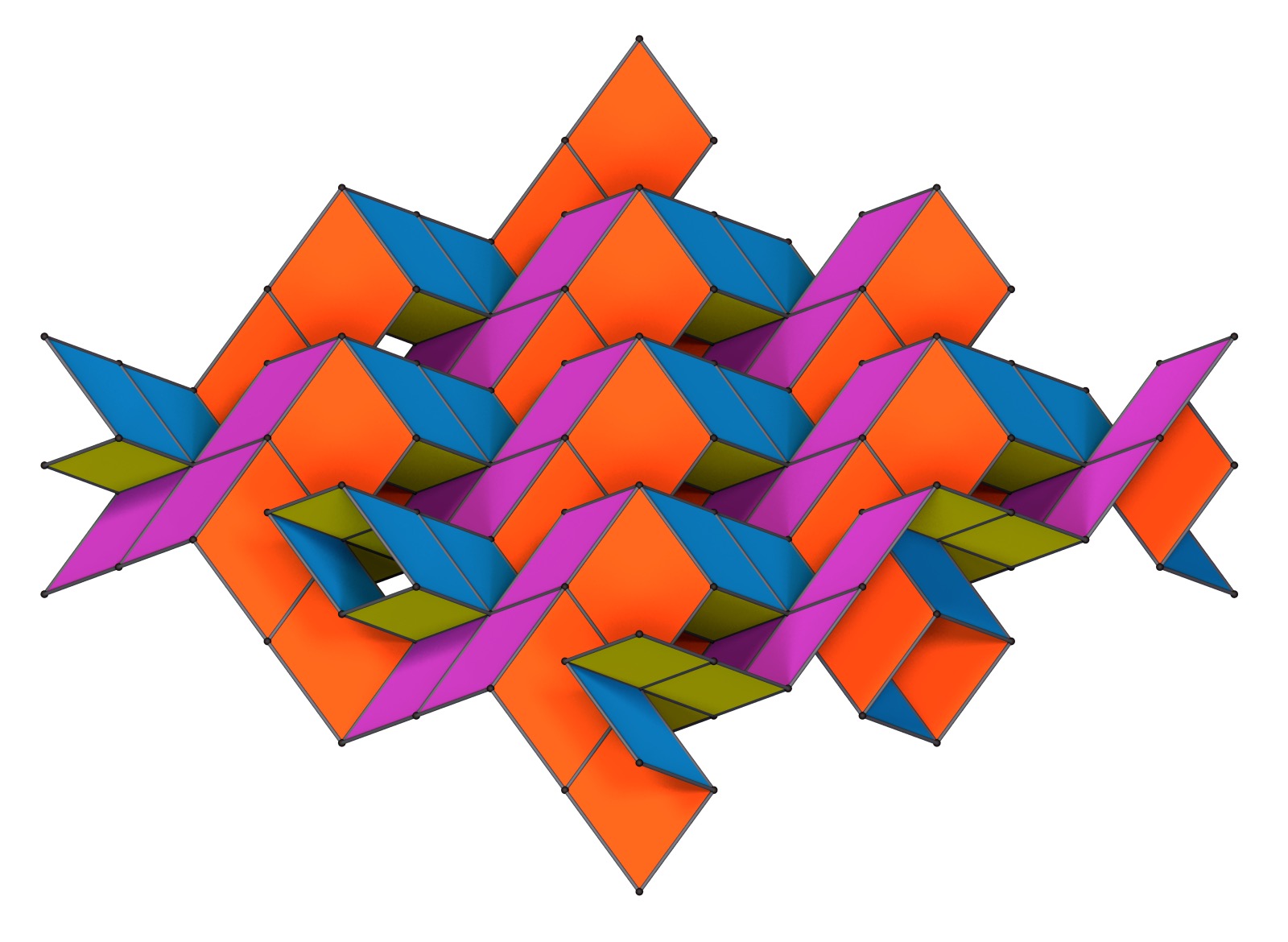}}
      \subcaptionbox{Compressed state\label{fig:MiuraWeave-4}}
        {\includegraphics[width=0.49\textwidth]{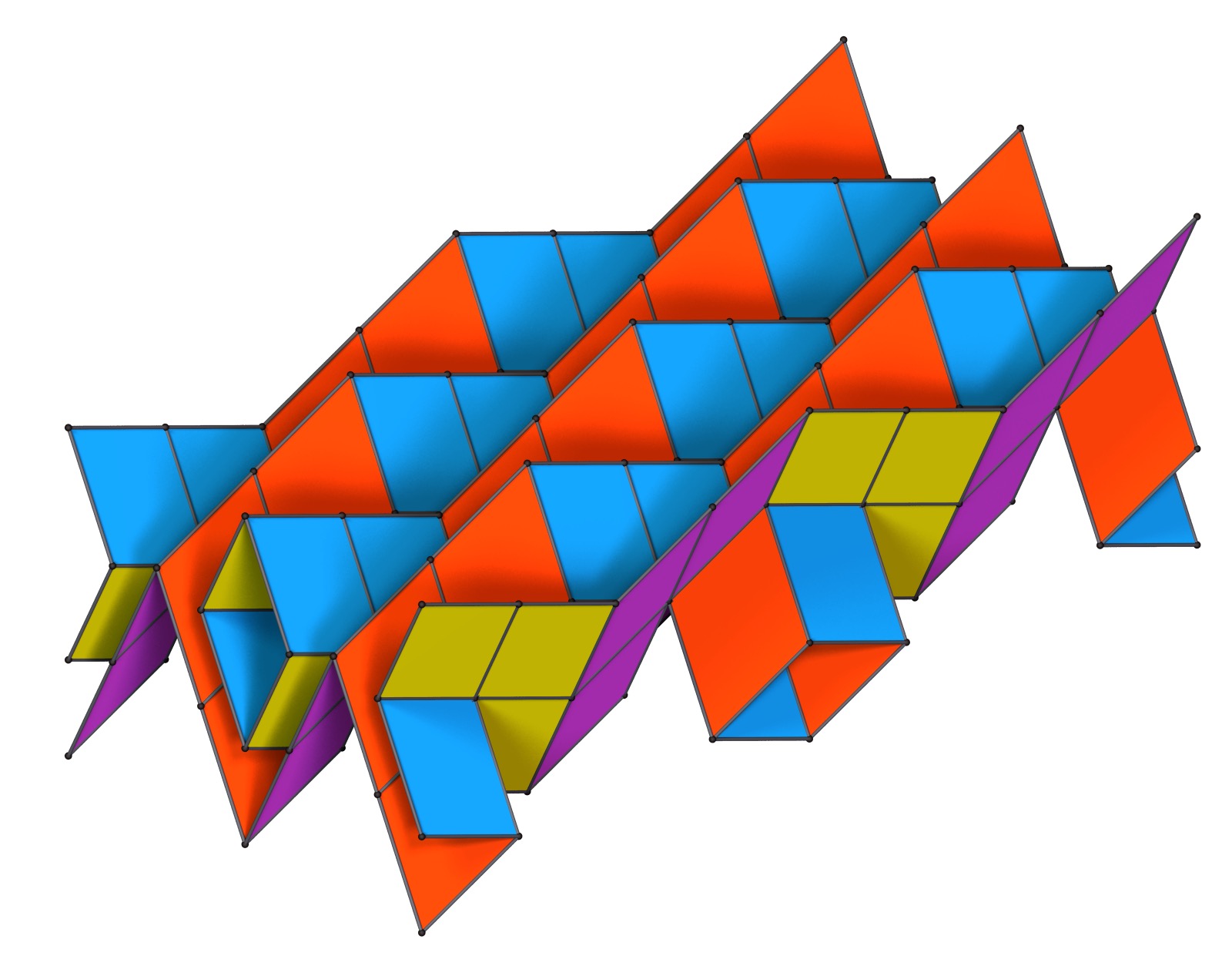}}
      \caption{Folding States of the Miura Weave}
      \label{fig:MiuraWeave}
    \end{figure}

\section{Doubly and Triply Periodic Links}
\label{sec:links}

In this section, we describe examples of doubly and triply periodic $\Sigma$-polyhedra. The basic building block consists of  two different hollowpeds  sharing an admissible facet. For the sake of concreteness and reproducibility, we use the hollowpeds $H_2$ and $H_4$ with common facet $\Pi_{13}$, which we call the {\em link}, see Figure \ref{fig:linksingle}. Note that the link only utilizes three of the four admissible facets, all except for $\Pi_{24}$.

\begin{figure}[H]
      \centering
      \subcaptionbox{the link\label{fig:linksingle}}
        {\includegraphics[width=0.49\textwidth]{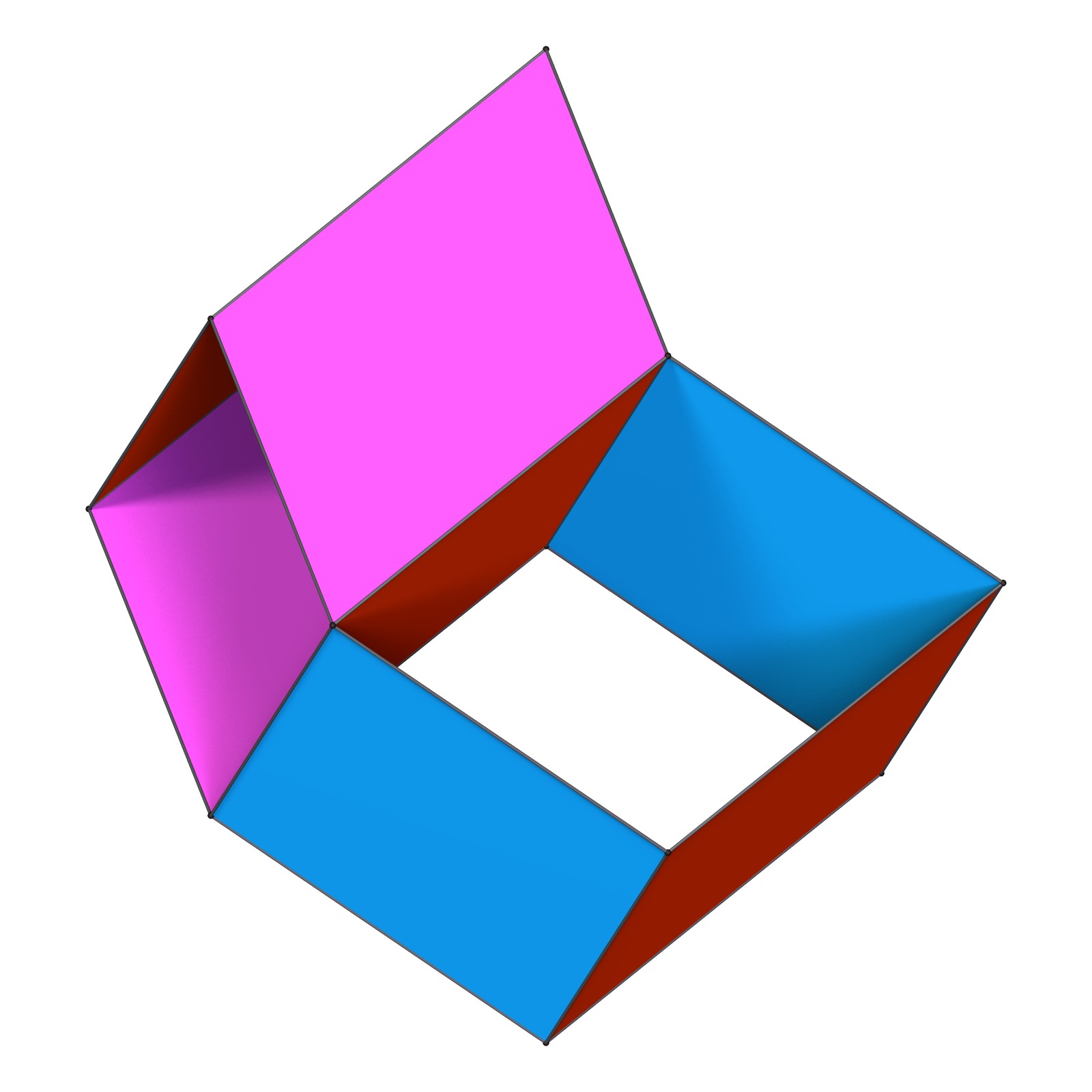}}
      \subcaptionbox{translated copies\label{fig:linktriply}}
        {\includegraphics[width=0.49\textwidth]{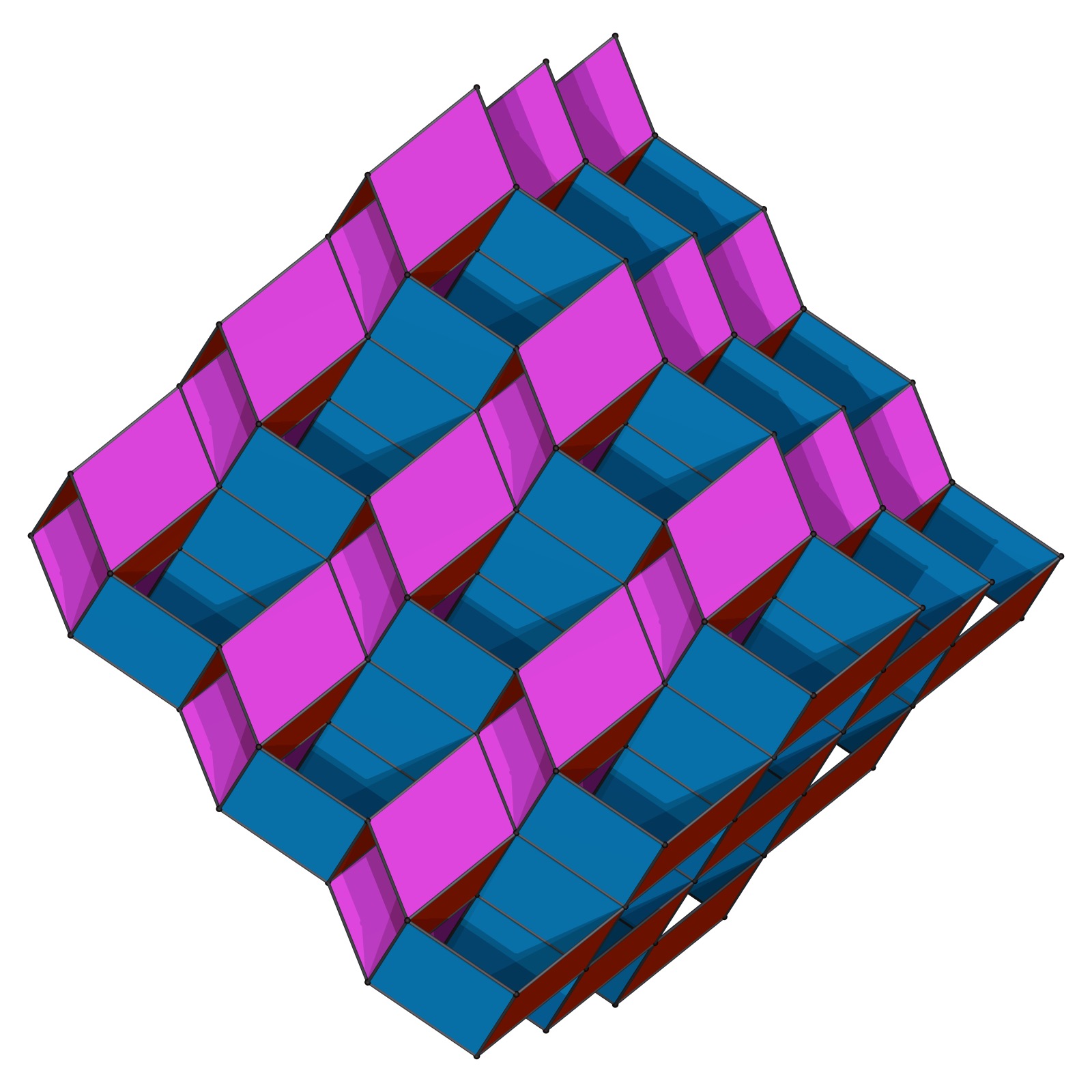}}
      \caption{The triply periodic link}
      \label{fig:link}
    \end{figure}

Let $\Lambda$ be the lattice spanned by $v_1-v_3$, $v_2-v_4$ and $v_1+v_3$. Then translating the link by lattice vectors produces a triply periodic bifoldable polyhedron supported by $\Sigma$, see Figure \ref{fig:linktriply}. The lack of the fourth facet further increases the flexibility of this polyhedron, which might be an undesirable feature.

The genus of a  triply periodic surface is commonly defined as the genus of the quotient by the maximal group of orientation preserving translations that leave the surface invariant. Note that in this case, the translation by $v_2-v_4$ is orientation reversing, so two links are needed for a fundamental domain.
The quotient surface then has eight vertices with valency 6, four of type \ref{fig:AcuteX} and four of type \ref{fig:ObtuseX}, with total curvature $-16\pi$ so that the genus is 5.
 
\medskip

There is a more interesting variation of this construction that also employs the fourth facet $\Pi_{24}$: We first combine a link and its mirror image (which also consists of $\Pi_1$  and $\Pi_3$ but translated differently) and connect them with copy of the facet $\Pi_{24}$. We call this piece the {\em butterfly}:

\begin{figure}[H]
      \centering
      \subcaptionbox{The butterfly\label{fig:Butterfly}}
        {\includegraphics[width=0.49\textwidth]{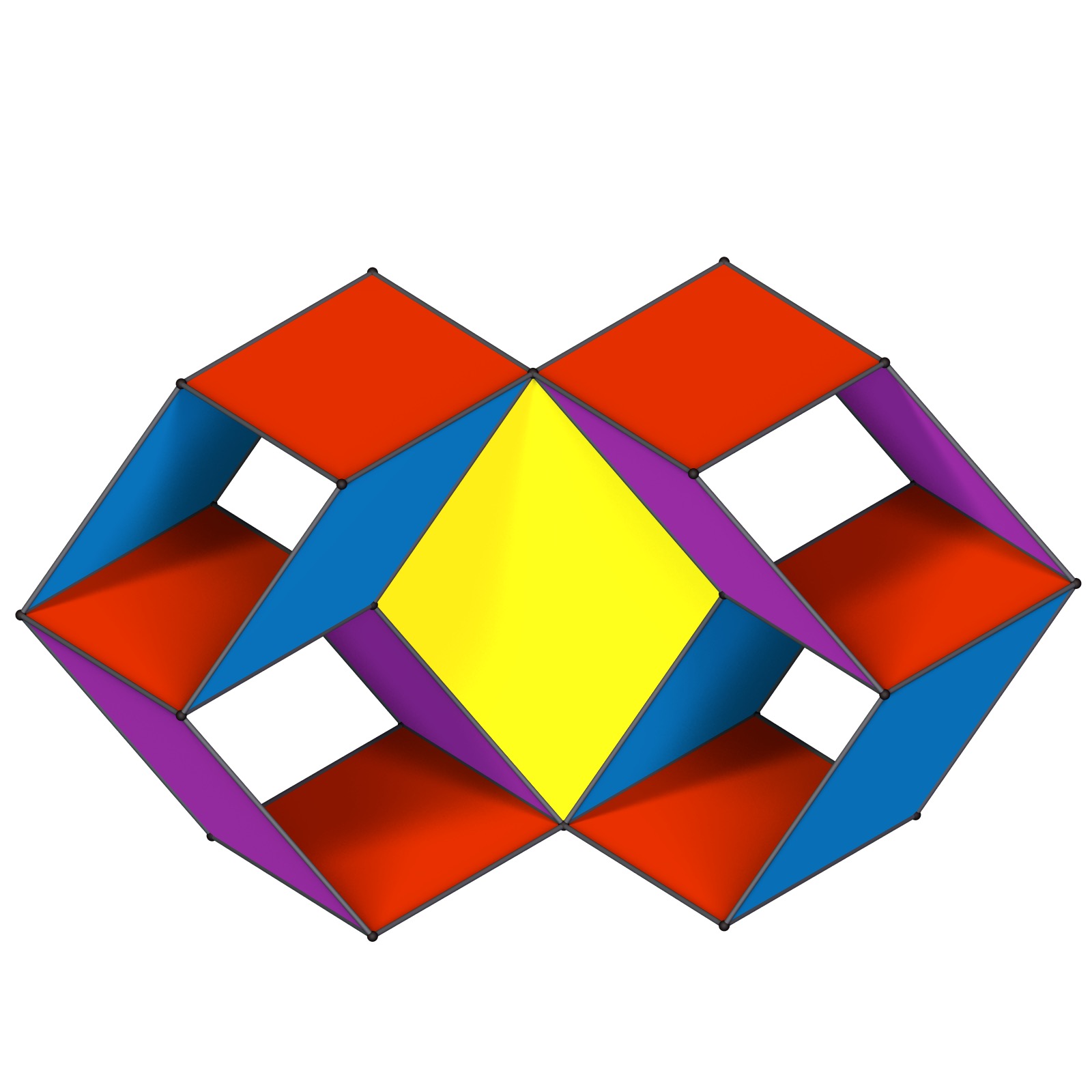}}
      \subcaptionbox{A linked dodecahedron\label{fig:ButterflyDodeca}}
        {\includegraphics[width=0.49\textwidth]{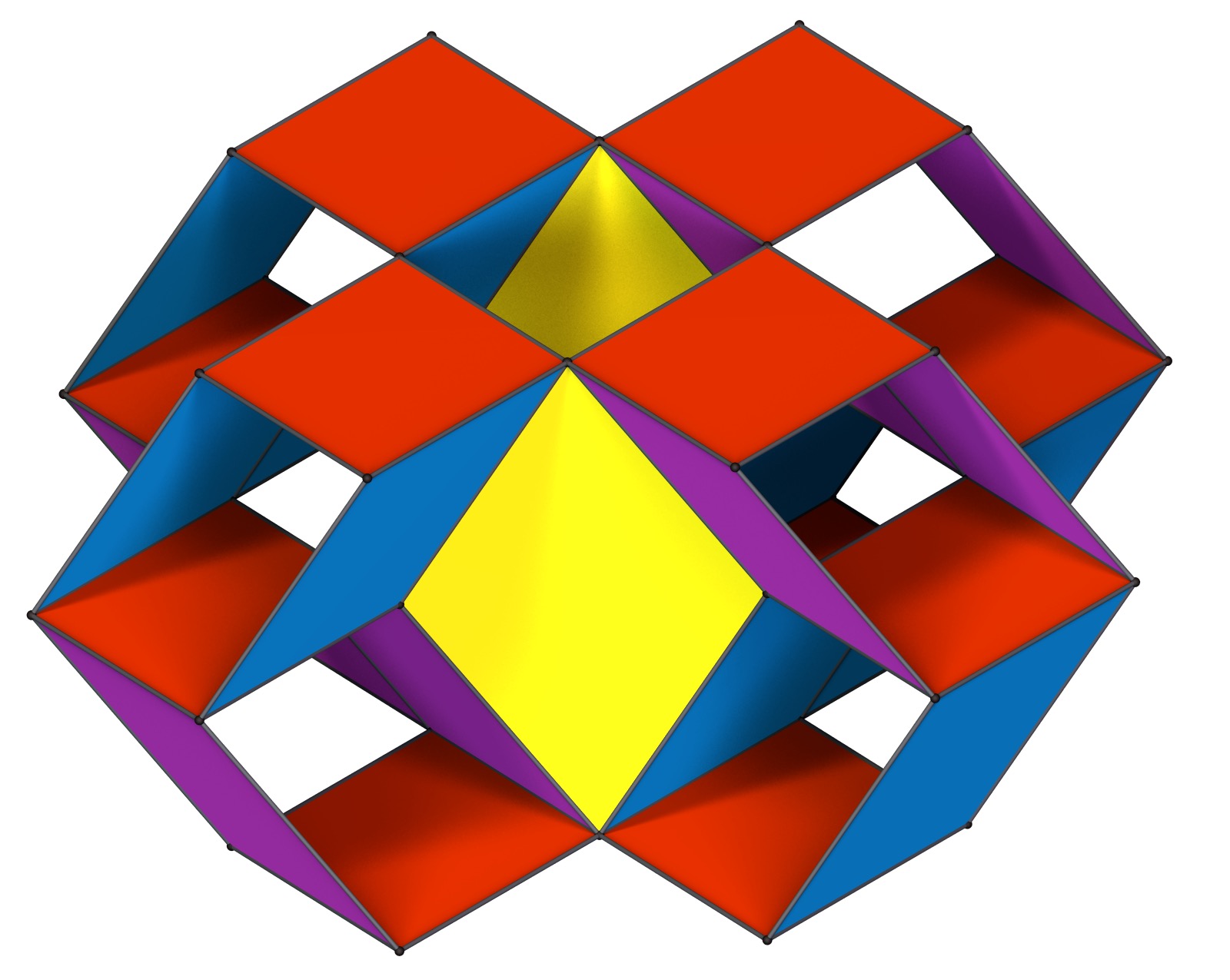}}
      \caption{The triply periodic butterfly}
      \label{fig:triplebutter}
    \end{figure}
    
 Translating the butterfly by $v_1-v_3$ leaves room for a standard dodecahedron between the four links.

The butterfly can be translated further to create the following triply periodic example:
\begin{figure}[H]
   \centering
   \includegraphics[width=4in]{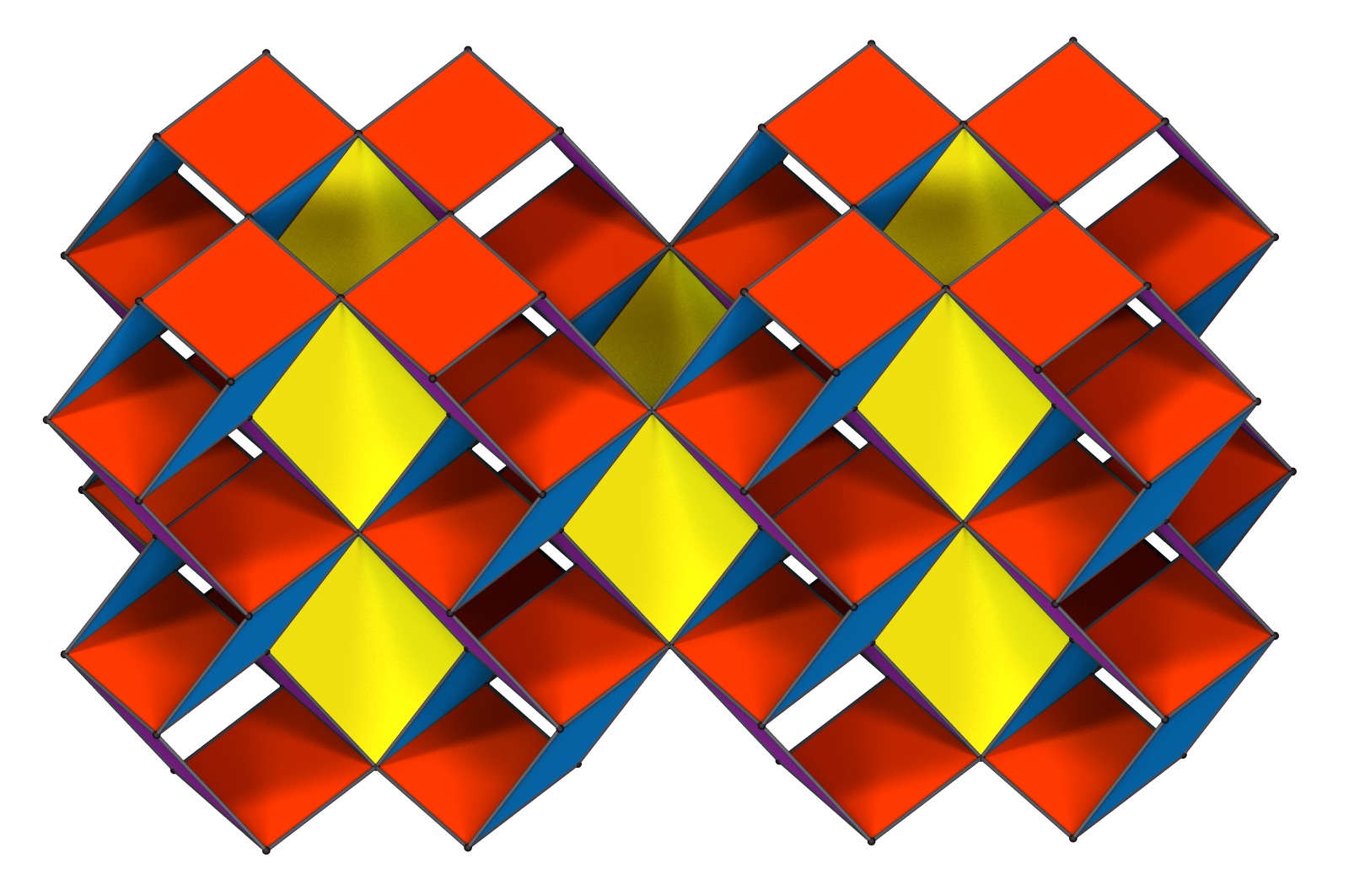}
   \caption{The triply periodic butterfly}
   \label{fig:ButterflyTriply}
\end{figure}

This surface has also genus 5. One can see this either similar as above, or by looking at the standard dodecahedra which have six facets missing. Two of them appear in the quotient, and the missing facets are used to form six handles, giving genus 5, as one of them is used to connect the dodecahedra.

Finally, in Figure \ref{fig:LinkDoubly} is a simple doubly periodic $\Sigma$-polyhedron where links are first translated into chains that are then connected by two ``snakes'' of parallelograms.

\begin{figure}[H]
   \centering
   \includegraphics[width=2.5in]{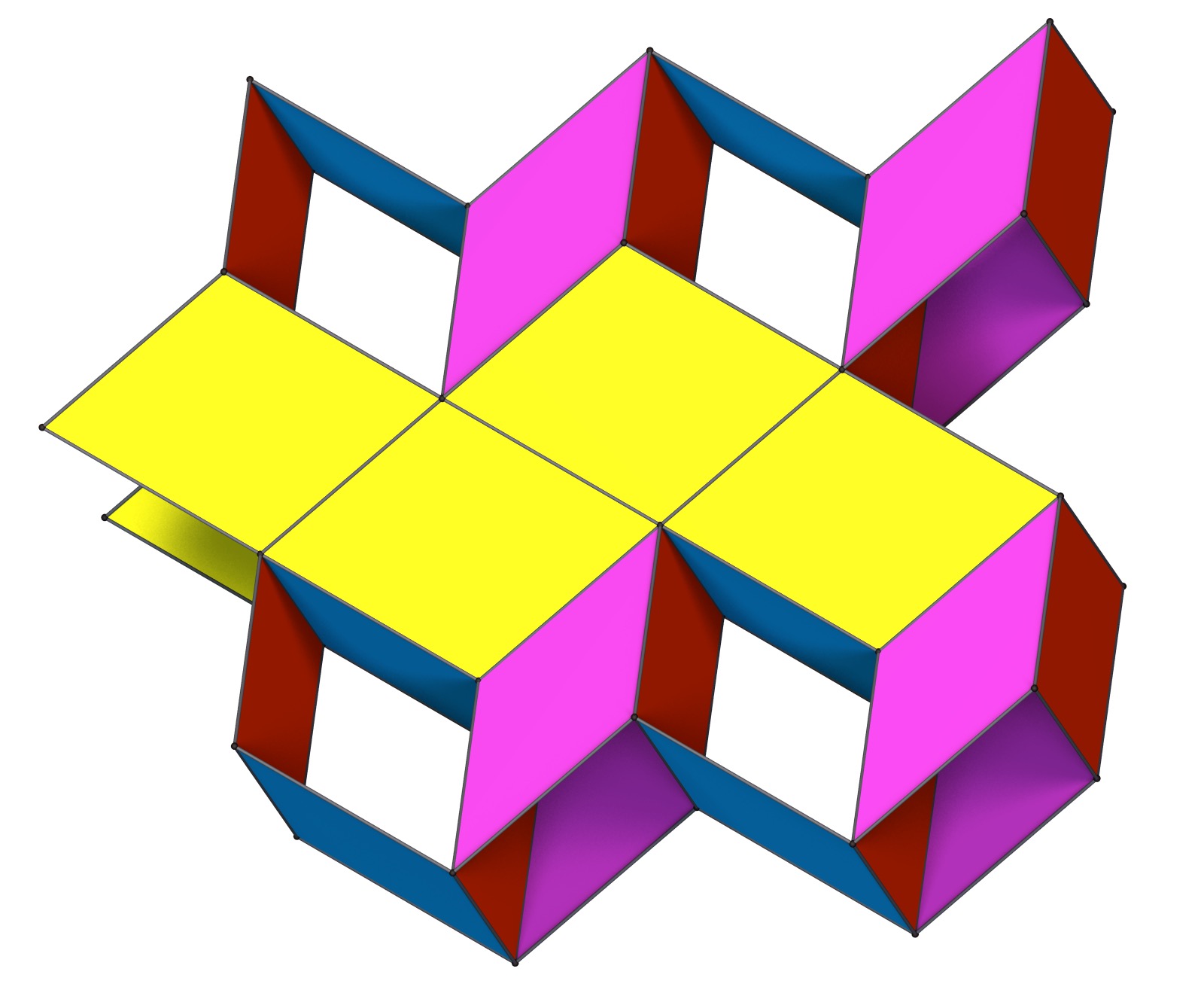}
   \caption{A doubly periodic $\Sigma$-polyhedron built with links}
   \label{fig:LinkDoubly}
\end{figure}

\section{The Dos Equis Pattern}
\label{sec:dosequis}

In this section, we construct a doubly periodic bifoldable $\Sigma$-polyhedron with boundary, called a ``layer''. Two such layers can be combined  in two different ways by parallel translation. Repeating this  allows to construct an infinitude of doubly periodic polyhedral surfaces without boundary. If the combination of layers is carried out periodically, the result will be a triply periodic $\Sigma$-polyhedron. In  its simplest form, this surface has the appearance of a triply periodic Miura pattern.

We begin with a vertex of valency eight arranged in an $X$-shape, see Figure \ref{fig:X}. On top of it we put a mirror image of the $X$. The result is clearly  translation invariant in the vertical direction, see Figure \ref{ig:dosequis}.

\begin{figure}[H]
   \centering
   \includegraphics[width=2.5in]{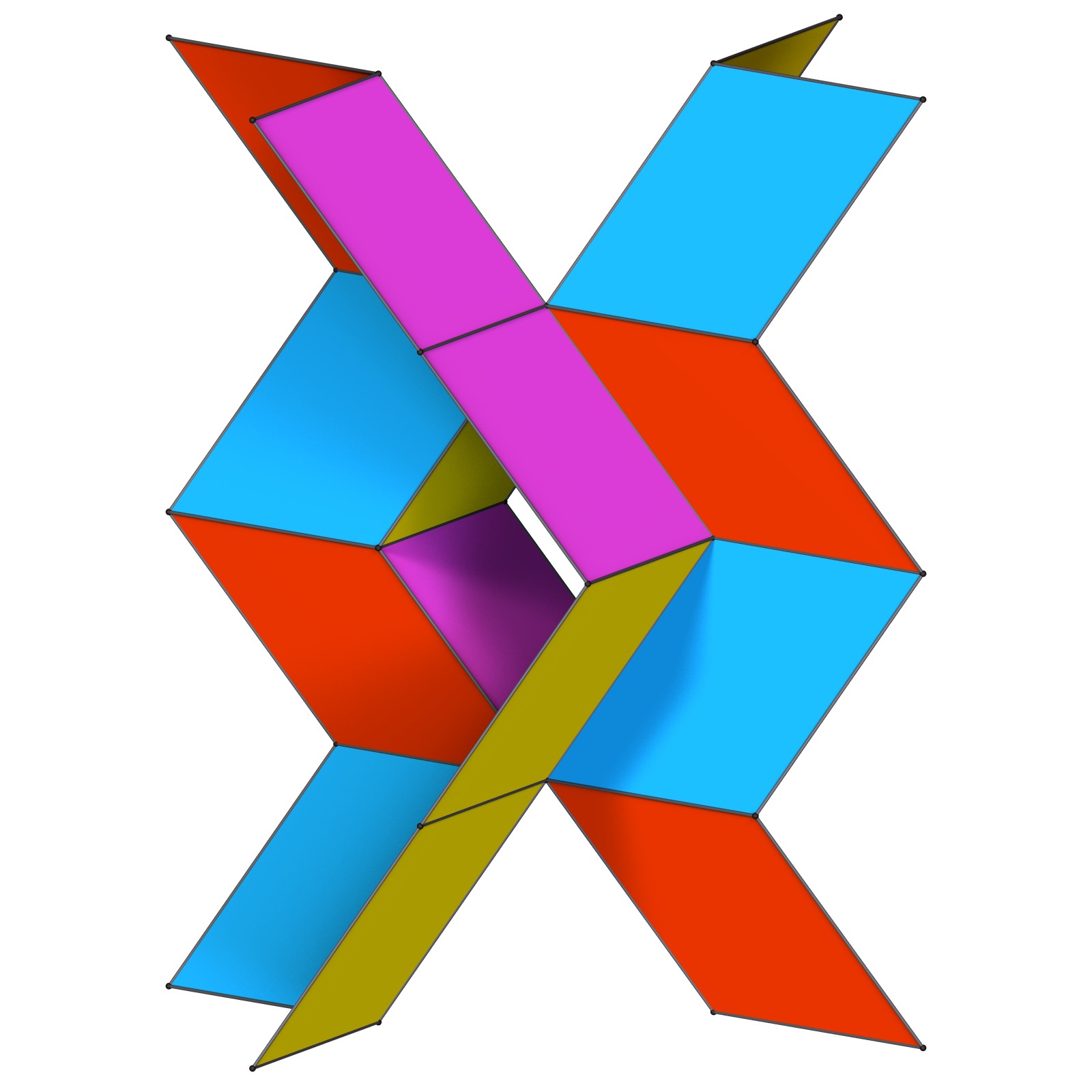}
   \caption{Two stacked X}
   \label{ig:dosequis}
\end{figure}

We add a second such double $X$ with the order of the two $X$s switched and obtain a translational fundamental piece of the doubly periodic layer.   
In Figure \ref{fig:XX3}, the translational directions are up/down and forward/backward.
\begin{figure}[H]
      \centering
      \subcaptionbox{$\alpha=24^\circ$}
        {\includegraphics[width=0.3\textwidth]{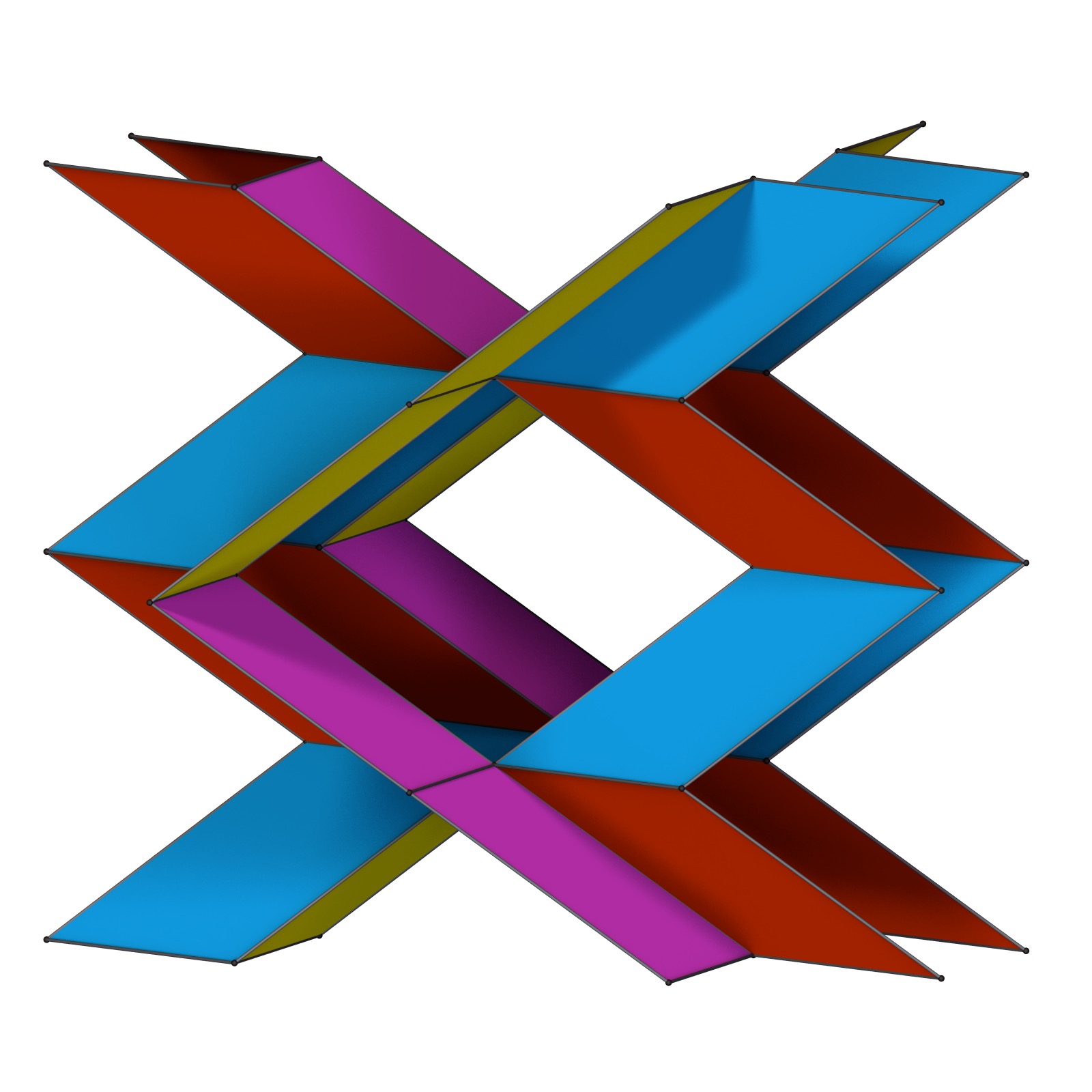}}
      \subcaptionbox{$\alpha=35^\circ$}
        {\includegraphics[width=0.3\textwidth]{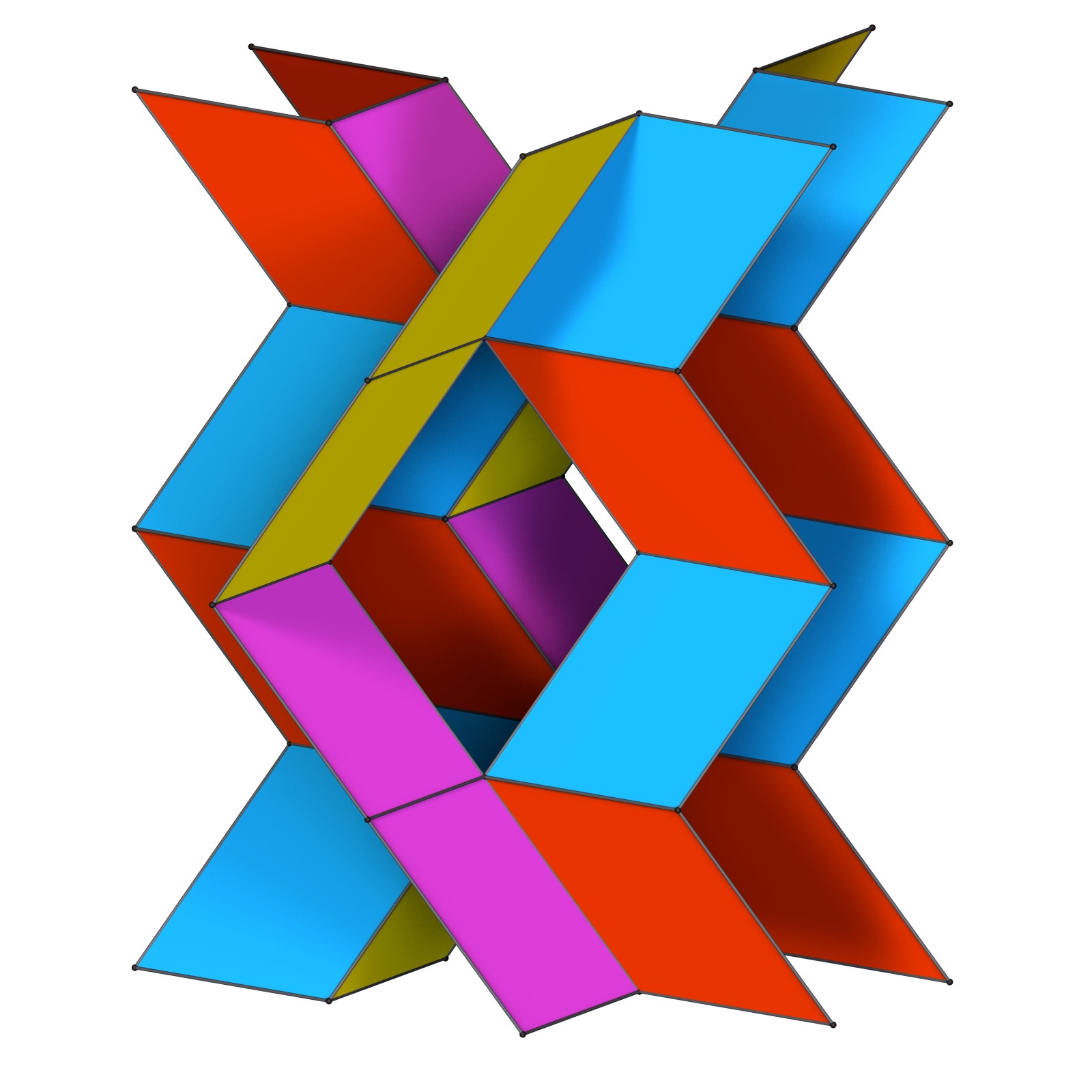}}
        \subcaptionbox{$\alpha=65^\circ$\label{fig:TripleX3carrow}}
        {\includegraphics[width=0.3\textwidth]{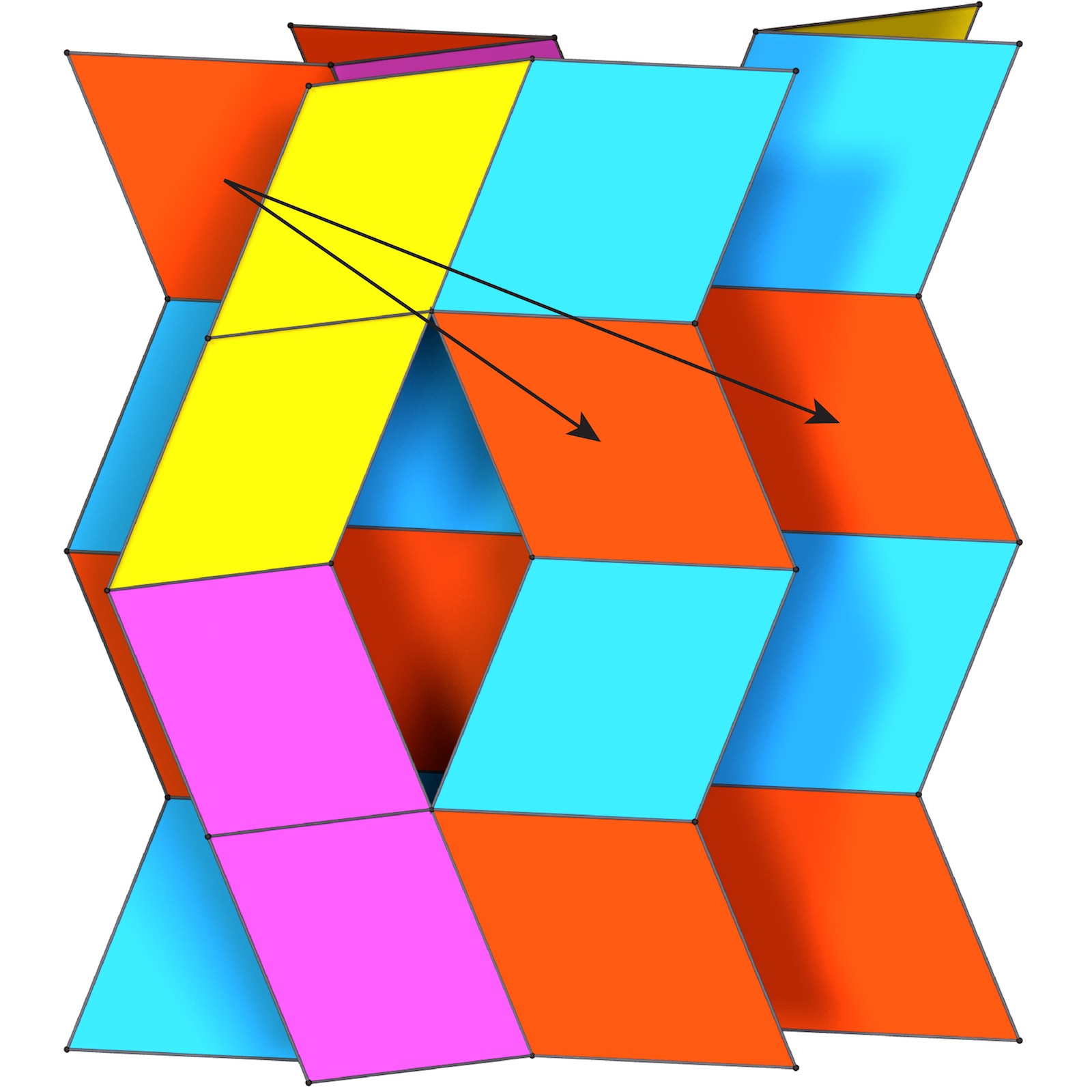}}
      \caption{Three states of the Dos Equis layer}
      \label{fig:XX3}
    \end{figure}    
    
  We see now that there are two different ways to identify one such layer with a translational copy (left/right) as indicated by the arrows in Figure \ref{fig:TripleX3carrow}.  One of these translations is orientation preserving, the other reverses orientation. Repeating this allows for an (uncountable) number of possible surfaces.

%
%

\bibliography{foldings}

\newcommand{\etalchar}[1]{$^{#1}$}
\begin{thebibliography}{CTCM14}

\bibitem[Alf68]{alfred1968}
Brunner Alfred.
\newblock Expansible surface structure, January~9 1968.
\newblock US Patent 3,362,118.

\bibitem[BBA{\etalchar{+}}11]{Barbarino2011}
Silvestro Barbarino, Onur Bilgen, Rafic~M. Ajaj, Michael~I. Friswell, and
  Daniel~J. Inman.
\newblock A review of morphing aircraft.
\newblock {\em Journal of Intelligent Material Systems and Structures},
  22(9):823--877, 2011.

\bibitem[Cox38]{coxeter1938}
HSM Coxeter.
\newblock Regular skew polyhedra in three and four dimension, and their
  topological analogues.
\newblock {\em Proceedings of the London Mathematical Society}, 2(1):33--62,
  1938.

\bibitem[Cox73]{coxeter1973}
Harold Scott~Macdonald Coxeter.
\newblock {\em Regular polytopes}.
\newblock Courier Corporation, 1973.

\bibitem[CTCM14]{cheung2014}
Kenneth~C Cheung, Tomohiro Tachi, Sam Calisch, and Koryo Miura.
\newblock Origami interleaved tube cellular materials.
\newblock {\em Smart Materials and Structures}, 23(9):094012, 2014.

\bibitem[FTP15]{filipov2015}
Evgueni~T Filipov, Tomohiro Tachi, and Glaucio~H Paulino.
\newblock Origami tubes assembled into stiff, yet reconfigurable structures and
  metamaterials.
\newblock {\em Proceedings of the National Academy of Sciences},
  112(40):12321--12326, 2015.

\bibitem[GB10]{Grosso2010}
A~E~Del Grosso and P~Basso.
\newblock Adaptive building skin structures.
\newblock {\em Smart Materials and Structures}, 19(12):124011, 2010.

\bibitem[HAB{\etalchar{+}}10]{Hawkes2010}
E.~Hawkes, B.~An, N.~M. Benbernou, H.~Tanaka, S.~Kim, E.~D. Demaine, D.~Rus,
  and R.~J. Wood.
\newblock Programmable matter by folding.
\newblock {\em Proceedings of the National Academy of Sciences},
  107(28):12441--12445, 2010.

\bibitem[KTY{\etalchar{+}}06]{Kuribayashi2006}
Kaori Kuribayashi, Koichi Tsuchiya, Zhong You, Dacian Tomus, Minoru Umemoto,
  Takahiro Ito, and Masahiro Sasaki.
\newblock Self-deployable origami stent grafts as a biomedical application of
  ni-rich tini shape memory alloy foil.
\newblock {\em Materials Science and Engineering: A}, 419(1):131 -- 137, 2006.

\bibitem[Miu69]{miura1969}
Koryo Miura.
\newblock {\em Proposition of pseudo-cylindrical concave polyhedral shells}.
\newblock Institute of Space and Aeronautical Science, University of Tokyo,
  1969.

\bibitem[SG13]{Schenk2013}
Mark Schenk and Simon~D. Guest.
\newblock Geometry of miura-folded metamaterials.
\newblock {\em Proceedings of the National Academy of Sciences},
  110(9):3276--3281, 2013.

\bibitem[SVSG14]{Schenk2014}
Mark Schenk, Andrew~D. Viquerat, Keith~A. Seffen, and Simon~D. Guest.
\newblock Review of inflatable booms for deployable space structures: Packing
  and rigidization.
\newblock {\em Journal of Spacecraft and Rockets}, 51(3):762--778, 2014.

\bibitem[TM12]{Tachi2012}
Tomohiro Tachi and Koryo Miura.
\newblock Rigid-foldable cylinders and cells.
\newblock {\em J. Int. Assoc. Shell Spat. Struct.}, 53:217--226, 2012.

\bibitem[Wu18]{Wu2018}
Jiangmei Wu.
\newblock Folding yoshimura pattern into large-scale art installation.
\newblock In {\em Origami 7: Proceedings of the 7th International Meeting on
  Origami in Science, Mathematics and Education}. University of Oxford, 2018.

\end{thebibliography}
\bibliographystyle{alpha}

\end{document}